\documentclass[11pt,a4paper,reqno]{amsart}
\usepackage[english]{babel}
\usepackage{pst-grad} % For gradients
\usepackage{pst-plot} % For axes
\usepackage{pstricks}
\usepackage{amsmath,amssymb,mathrsfs,amsthm}
\usepackage{tikz} 
\usepackage{mathcomp,wasysym}  
\usepackage{graphicx}  
\usepackage[all]{xy} \xyoption{arc} \xyoption{color}
\usepackage{epsfig}
\usepackage[a4paper,
left=2.5cm, right=2.5cm,
top=3cm, bottom=3cm]{geometry}
\newcommand{\RR}{\mathbb R}

\newcommand{\NN}{\mathbb N}
\newcommand{\ZZ}{\mathbb Z}

\newcommand{\TT}{\mathbb T}
\newcommand{\pat}{\partial_t}
\newcommand{\pax}{\partial_x}

\newcommand{\warrow}{\rightharpoonup}

\newcommand{\vertiii}[1]{{\left\vert\kern-0.25ex\left\vert\kern-0.25ex\left\vert #1 
    \right\vert\kern-0.25ex\right\vert\kern-0.25ex\right\vert}}

\newcommand{\F}{\bar{f}}
\newcommand{\G}{\bar{g}}

\normalsize
\normalsize
\newtheorem{satz}{Proposition}[section]
\newtheorem{lem}[satz]{Lemma} 
\newtheorem{remark}[satz]{Remark}
\newtheorem{theorem}{Theorem}
\newtheorem{definition}{Definition}

\definecolor{luh-dark-blue}{rgb}{0.0, 0.313, 0.608}

\usepackage[colorlinks=true]{hyperref}
\hypersetup{urlcolor=luh-dark-blue, citecolor=luh-dark-blue, linkcolor=luh-dark-blue}

\def\jump#1{{[\hspace{-2pt}[#1]\hspace{-2pt}]}}
\title{On the thin film Muskat and the thin film Stokes equations}

\author[G. Bruell]{Gabriele Bruell}
\email{gabriele.bruell@ntnu.no}
\address{Department of Mathematical Sciences, Norwegian University of Science and Technology, 7491 Trondheim, Norway}

\author[R. Granero-Belinch\'{o}n]{Rafael Granero-Belinch\'{o}n}
\email{rafael.granero@unican.es}
\address{Departamento  de  Matem\'aticas,  Estad\'istica  y  Computaci\'on,  Universidad  de Cantabria.  Avda.  Los  Castros  s/n,  Santander,  Spain.}

\begin{document}
\begin{abstract}
The present paper is concerned with the analysis of two strongly coupled systems of degenerate parabolic partial differential equations arising in multiphase thin film flows. In particular, we consider the two-phase thin film Muskat problem and the two-phase thin film approximation of the Stokes flow under the influence of both, capillary and gravitational forces. The existence of global  weak solutions for \emph{medium size} initial data in \emph{large function spaces} is proved. Moreover, exponential decay results towards the equilibrium state are established, where the decay rate can be estimated by explicit constants depending on the physical parameters of the system. Eventually, it is shown that if the initial datum satisfies additional (low order) Sobolev regularity, we can propagate Sobolev regularity for the corresponding solution. The proofs are based on a priori energy estimates in Wiener and Sobolev spaces.

\end{abstract}


\maketitle
{\small
\tableofcontents}

\allowdisplaybreaks
\section{Introduction}

The dynamics of viscous thin fluid films is a widely studied topic in the area of fluid dynamics. A classical approach to gain insight in the evolutionary behavior of thin fluid films is to apply lubrication approximation and cross sectional averaging to the governing equations, which leads to simplified model equations. Considering thin films it is instinctive  that surface tension effects play a significant role. A common feature of many thin film approximations is that the presents of surface tension leads to fourth-order equations. Due to the degenerate character of the equations, it is not to be expected that classical solutions exist globally in time, unless the initial datum is close to a stable steady state. Pioneering works on the existence of global weak solutions of the classical thin film equation and their properties are due to Bernis \& Friedman \cite{Bernis1990} followed by Beretta, Bertsch \& Dal Passo \cite{Beretta1995} and Bertozzi \& Pugh  \cite{Bertozzi1996}. Since then, the study of thin film equations attracted a lot of attention and many authors contributed to a deeper understanding with respect to several aspects of the underlying mechanisms.

\medskip

The concern of the present work is the existence of global weak solutions for two parabolic, strongly coupled and degenerated systems arising as a thin film approximation: the thin film Muskat problem modeling a two-phase flow in porous medium and the thin film Stokes problem that arises as a model of a two-phase flow for highly viscous Newtonian fluids. 
Both, the Muskat and the Stokes problem share the same scenario: the fluid with label $``-"$ (i.e. whose velocity, pressure, viscosity, and density are $u_-,p_-,\mu_-$, and $\rho_-$, respectively) lies between the free boundary $f=f(x,t)$ and an impervious flat bottom, while the fluid with label $``+"$ (i.e. whose velocity, pressure, viscosity, and density are $u_+,p_+,\mu_+$, and $\rho_+$, respectively) is between the free surface $h=h(x,t)$ and the internal wave $f$. Over the top fluid we have air that is assumed to behave like vacuum. In other words, the (common) domains that we consider in this paper can be described as $\Omega(t)=\overline{\Omega_+(t)}\cup \overline{\Omega_-(t)}$, where
\begin{align*}
\Omega_+(t)&=\{(x,y)\in I\times \RR,\;\; f(x,t)<y<h(x,t)\},\\
\Omega_-(t)&=\{(x,y)\in I\times \RR,\;\; 0<y<f(x,t)\},
\end{align*}
and the functions $f,h$ satisfy
$$
h(x,t)>f(x,t)>0.
$$
Here, $I$ denotes the domain of the horizontal variable. 
\begin{figure}[h]\label{fig1}
\begin{center} 
\begin{tikzpicture}[domain=0:3*pi, scale=0.9] 
\draw (pi,2.5) node { air};
\draw[ultra thick, smooth, variable=\x, luh-dark-blue] plot (\x,{0.3*cos(\x r)+2}); 
\fill[luh-dark-blue!10] plot[domain=0:3*pi] (\x,{0.2*sin(\x r)+1}) -- plot[domain=3*pi:0] (\x,{0.3*cos(\x r)+2});
\draw[ultra thick, smooth, color=red] plot (\x,{0.2*sin(\x r)+1});
\fill[red!10] plot[domain=0:3*pi] (\x,0) -- plot[domain=3*pi:0] (\x,{0.2*sin(\x r)+1});
\draw[very thick,<->] (3*pi+0.4,0) node[right] {$x$} -- (0,0) -- (0,3.5) node[above] {$z$};
\node[right] at (3*pi,1.7) {$h(x,t)$};
\node[right] at (3*pi,0.9) {$f(x,t)$};
\node[right] at (1,0.5) {$\Omega_-(t)$};
\node[right] at (1,1.6) {$\Omega_+(t)$};
 \draw (2*pi,1.5) node { $p_+,u_+, \rho_+,\mu_+$}; 
    \draw (2*pi,0.4) node { $p_-,u_-,\rho_-,\mu_-$}; 
\draw[-] (0,-0.3) -- (0.3, 0);
\draw[-] (0.5,-0.3) -- +(0.3, 0.3);
\draw[-] (1,-0.3) -- +(0.3, 0.3);
\draw[-] (1.5,-0.3) -- +(0.3, 0.3);
\draw[-] (2,-0.3) -- +(0.3, 0.3);
\draw[-] (2.5,-0.3) -- +(0.3, 0.3);
\draw[-] (3,-0.3) -- +(0.3, 0.3);
\draw[-] (3.5,-0.3) -- +(0.3, 0.3);
\draw[-] (4,-0.3) -- +(0.3, 0.3);
\draw[-] (4.5,-0.3) -- +(0.3, 0.3);
\draw[-] (5,-0.3) -- +(0.3, 0.3);
\draw[-] (5.5,-0.3) -- +(0.3, 0.3);
\draw[-] (6,-0.3) -- +(0.3, 0.3);
\draw[-] (6.5,-0.3) -- +(0.3, 0.3);
\draw[-] (7,-0.3) -- +(0.3, 0.3);
\draw[-] (7.5,-0.3) -- +(0.3, 0.3);
\draw[-] (8,-0.3) -- +(0.3, 0.3);
\draw[-] (8.5,-0.3) -- +(0.3, 0.3);
\draw[-] (9,-0.3) -- +(0.3, 0.3);
\draw (pi,-0.7) node { bottom};
\end{tikzpicture}   
\end{center}
\caption{The fluid-air interface $h$ and the fluid-fluid interface $f$.}
\end{figure}
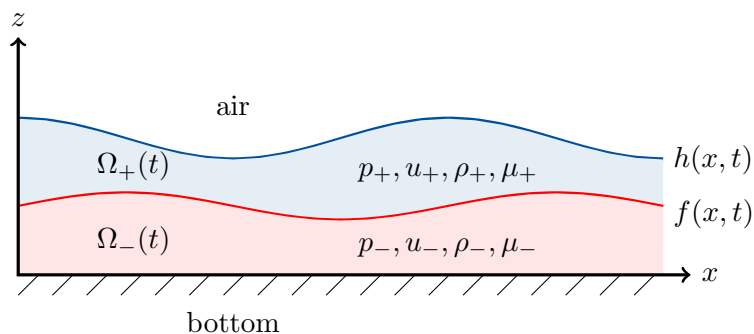

\medskip

\subsection{The thin film Muskat problem}

We are going to introduce the Muskat problem, a thin film approximation of the problem and some prior results. Moreover, we reformulate the thin film approximation in a way to be suitable for our subsequent study.

\subsubsection{The equations}
The Muskat problem reads
%\begin{subequation} %\label{MUSKAT}
\begin{alignat*}{2}
\mu_\pm u_\pm+\nabla p_\pm&=-\rho_\pm G e_2,  \qquad&&\text{in}\quad \Omega_\pm(t)\times[0,T]\,,\\
\nabla\cdot u_\pm &=0,  \qquad&&\text{in}\quad \Omega_\pm(t)\times[0,T]\,,\\
\jump{p} &= \gamma_f \mathcal{H}_{\Gamma(f)}\qquad &&\text{on }\Gamma(f)\times[0,T],\\
p_+ &= -\gamma_h \mathcal{H}_{\Gamma(h)}\qquad &&\text{on }\Gamma(h)\times[0,T],\\
\pat f &= u_\pm\cdot(-\pax f,1)\qquad &&\text{on }\Gamma(f)\times[0,T],\\
\pat h &= u_+\cdot(-\pax h,1)\qquad &&\text{on }\Gamma(h)\times[0,T],\\
u_-\cdot e_2 &= 0\qquad &&\text{on }\Gamma_0\times[0,T],
\end{alignat*}
%\end{subequations}
where $\Gamma_0:=\{y=0\}$ is the bottom and the fluid-fluid and fluid-air interfaces are located at $\Gamma(f):=\{y=f\}$ and $\Gamma (h):=\{y=h\}$, respectively. The constants $\mu_{\pm}$ and $\rho_{\pm}$ denote the viscosity and density of the lower and upper fluid, respectively. Moreover, $\gamma_f$ is the surface tension coefficient at the interface $\Gamma(f)$, while $\gamma_h$ is the surface tension coefficient at the interface $\Gamma (h)$. Eventually, $\mathcal{H}_{\Gamma(h)}$ and $\mathcal{H}_{\Gamma(f)}$ denote the curvature of the interfaces $\Gamma(h)$ and $\Gamma(f)$, respectively. The constant $G$ represents the gravitational acceleration and $\jump{f} = f^+ - f^-$ the jump of a function $f$ across $\Gamma(f)$. 
The Muskat problem appears as a model of geothermal reservoirs, aquifers or oil wells \cite{CF, Muskat:porous-media} and has received a lot of attention in the last years. We refer the interested reader for instance to \cite{ambrose2004well,ccfgl,CGSVfiniteslope,c-g07,escher2011generalized,e-m10,emw15,gancedo2014absence,GG,KK,pruess2016muskat,SCH}.
Under the assumption of \emph{small layer thickness}, Escher, Matioc \& Matioc \cite{escher2012modelling} applied lubrication approximation and cross sectional averaging to derive the following system of partial differential equations for the evolution of two thin films in a porous medium
\begin{align}\label{eq:system}
\begin{split}
\partial_t f &= -\partial_x \left[f \left(\mu_-^{-1}\gamma_h \partial_x^3h+ \mu_-^{-1}\gamma_f\partial_x^3 f-\mu_-^{-1}G(\rho_--\rho_+)\partial_x f-\mu_-^{-1}G\rho_+\partial_xh\right)\right], \\
\partial_t h&=-\partial_x\left[f \left(\mu_-^{-1}\gamma_h \partial_x^3h+ \mu_-^{-1}\gamma_f\partial_x^3 f-\mu_-^{-1}G(\rho_--\rho_+)\partial_x f-\mu_-^{-1}G\rho_+\partial_xh\right) \right. \\
&\qquad\quad\;\;
 \left.+(h-f)\left(\mu_+^{-1}\gamma_h \partial_x^3h - \mu_+^{-1}G\rho_+\partial_xh\right) \right]
 \end{split}
\end{align}
for $x\in I=(0,\pi)$ and $t>0$. The system \eqref{eq:system} is supplemented with initial conditions for $f$ and $h$:
\begin{equation}\label{eq:initial}
f(0,x)=f_0(x), \qquad h(0,x)=h_0(x),
\end{equation}
and no-flux boundary conditions
%\rg{some lines above the domain was $(0,L)$. I think that we want the BC to be on $(0,L=\pi)$.Then we extend the solution periodically.}
\begin{equation}\label{eq:boundary}
\partial_x f = \partial_x h =\partial_x^3 f = \partial_x^3 h =0, \qquad\mbox{at}\quad x=0\quad \mbox{and}\quad{x=\pi}.
\end{equation}
We assume that both layers initially have a positive thickness, that is
\begin{equation}\label{positivity}
h_0(x)>f_0(x)>0.
\end{equation}

\subsubsection{Prior results for the thin film Muskat problem}
Since Escher, Matioc \& Matioc \cite{escher2012modelling} derived the thin film Muskat problem \eqref{eq:system}, this system has been intensively studied. In the absence of surface tension effects ($\gamma_f=\gamma_h=0$), problem \eqref{eq:system} reduces to a system of second order. In this case local existence of classical solutions in $H^2$ and exponential stability (in the $H^2$ norm) of steady state solutions were determined \cite{escher2012modelling}. The proofs are based on semigroup theory and an energy functional given by
$$
\mathscr{E}_{\text{gravity}}(h,f)=\int |f|^2+R|h|^2dx,
$$
where $R$ is a positive constant depending on the physical parameters of the problem. In a subsequent work, Escher \& Matioc \cite{escher2013existence} studied the case when surface tension effects are taken into account and proved local existence and asymptotic stability of steady states for the fourth order system \eqref{eq:system} in the Sobolev space $H^4$. Moreover, the authors found that
\begin{equation}\label{eq:energyH}
\mathscr{H}(h,f)=\int f\log(f)-f+1+S\left[(h-f)\log(h-f)-h+f+1\right]dx
\end{equation}
 is an energy functional, where $S$ is a positive constant depending on the physical parameters of the problem,
and studied non-flat equilibria which exist under stabilizing surface tension effects and destabilizing stratification. Concerning global solutions, Matioc \cite{matioc2012non} proved the existence of nonnegative global weak solutions in $H^1$ for the (purely) capillary driven thin film Muskat problem using a priori estimates provided by the energy functionals \eqref{eq:energyH} and
$$
\mathscr{E}_{\text{capillary}}(h,f)=\int |\pax f|^2+T|\pax h|^2dx,
$$
where $T$ is a positive constant depending on the physical parameters of the problem.
Escher, Lauren\c{c}ot \& Matioc \cite{ELM11} proved the existence of nonnegative global weak $L^2$ solutions for the gravity driven thin film. In addition to the existence result, they also proved exponential convergence towards equilibria in $L^2$ norms.
In the case when the thin film Muskat problem is considered on $I=\mathbb{R}$, Lauren\c{c}ot \& Matioc \cite{laurenccot2013gradient,laurencot2014thin} observed that \eqref{eq:system} is a gradient flow for the functional $\mathscr{E}_{\text{gravity}}$ with respect to the 2-Wasserstein distance in the set of Borel probability measures on $\mathbb{R}$ with finite second moment. This observation allowed them to obtain the existence of global weak $L^2$ solutions. Furthermore, Lauren\c{c}ot \& Matioc \cite{laurenccot2017finite,laurencot2017self} proved the existence of self-similar profiles and the convergence of weak $L^2$ solutions towards them (at an unknown rate) and the finite speed of propagation of a certain family of weak solutions.

\subsubsection{Reformulation of the thin film Muskat problem}\label{sectionreformulation}
The original problem \eqref{eq:system} is posed on the interval $[0,\pi]$ with boundary conditions \eqref{eq:boundary}. However, instead of dealing with the interval $[0,\pi]$ and no-flux boundary conditions \eqref{eq:boundary}, we will generalize the problem to consider periodic functions over the interval $[-\pi,\pi]$. Let us explain this in further detail.
We denote by $\tilde{f}$ and $\tilde{h}$ the even extensions of the unknowns $f$ and $h$, i.e. for $f$ defined on $[0,\pi]$ with boundary conditions \eqref{eq:boundary}, we define
$$
\tilde{f}(x,t)=f(|x|,t) \qquad (x,t)\in[-\pi,\pi]\times \RR_+,
$$
and similarly for $\tilde{h}$. Subsequently, we will drop the tilde notation and write $f,h$ for the unknowns defined on $[-\pi,\pi]$. 
Furthermore, note that equation \eqref{eq:system} preserves the even symmetry. Thus, we can generalize the problem by abandoning the eveness assumption from the initial data and seek for $2\pi$-periodic solutions $f,h$ to \eqref{eq:system}.
In order to recover the physically motivated problem posed on $[0,\pi]$ with boundary conditions \eqref{eq:boundary}, it is sufficient to consider an even periodic initial datum over $[-\pi,\pi]$ and to restrict the corresponding solution to the interval $[0,\pi]$.
% By uniqueness (that holds for smooth enough solutions), this restriction of our periodically extended problem should match the smooth solution of the original problem \eqref{eq:system} over $[0,\pi]$ with boundary conditions \eqref{eq:boundary}.

\medskip

Defining the new unknown $g:=h-f$, system \eqref{eq:system} can be rewritten as
\begin{align}\label{eq:system0}
\begin{split}
\partial_t f &= -\partial_x \left[f \left(\mu_-^{-1}\gamma_h \partial_x^3(g+f)+ \mu_-^{-1}\gamma_f\partial_x^3 f-\mu_-^{-1}G(\rho_--\rho_+)\partial_x f-\mu_-^{-1}G\rho_+\partial_x(g+f)\right)\right], \\
\partial_t g&=-\partial_x\left[g\left(\mu_+^{-1}\gamma_h \partial_x^3(g+f) - \mu_+^{-1}G\rho_+\partial_x(g+f)\right) \right]. 
 \end{split}
\end{align}
%\rg{Please, confirm that this is correct}
Set
\[
	b:=\rho_+,\qquad b_\mu := \frac{\mu_-}{\mu_+}b,\qquad b_\rho =\frac{\rho_-}{\rho_+}b,
\]
and, if $\gamma_h>0$,
$$
A=\frac{\gamma_h}{G}, \qquad A_\mu := \frac{\mu_-}{\mu_+}A,\qquad A_{\gamma}:=\frac{\gamma_f+\gamma_h}{\gamma_h}A.
$$
Introducing a new time variable
\[
	\tilde t :=G\mu^{-1}_-t,
\]
and suppressing thereafter the tildes, \eqref{eq:system0} reduces to
\begin{align}\label{eq:system1}
\begin{split}
\partial_t f &= -\partial_x \left[f \left(A_\gamma \partial_x^3 f + A\partial_x^3g-b_\rho \partial_xf - b\partial_xg\right) \right], \\
\partial_t g&=-\partial_x\left[g\left(A_\mu \partial_x^3f + A_\mu \partial_x^3 g - b_\mu \partial_x f - b_\mu \partial_x g\right)\right].
 \end{split}
\end{align}
Note that \eqref{positivity} translates into
$$
g_0(x):=h_0(x)-f_0(x)>0,\qquad f_0(x)>0.
$$

Let us define the mean of a function $f$ as
\[
	\langle f \rangle := \frac{1}{2\pi}\int_{-\pi}^\pi f(x)\, dx.
\]
We have the following quick observation:
\begin{lem}[Conservation of mass for \eqref{eq:system1}]
Let $(f,g)$ be a smooth solution to \eqref{eq:system1} on $[0,T)$, then the mass of $f$ and $g$ is preserved in time, that is
\[
	\langle f(t) \rangle = \langle f_0 \rangle \quad \mbox{and} \quad	\langle g(t) \rangle = \langle g_0 \rangle \qquad \mbox{for all}\quad t \in [0,T).
\]
\end{lem}

In spirit of the previous lemma, we introduce the zero mean functions
\begin{equation}\label{newvariables}
	\F := f-\langle f_0 \rangle \qquad \mbox{and} \qquad \G := g - \langle g_0 \rangle
\end{equation}
and obtain that
\begin{align}\label{eq:system2}
\begin{split}
\partial_t \F &= -\partial_x \left[(\F+\langle f_0 \rangle) \left(A_\gamma \partial_x^3 \F + A\partial_x^3\G-b_\rho \partial_x\F - b\partial_x\G\right)\right], \\
\partial_t \G&=-\partial_x\left[ (\G+\langle g_0 \rangle)\left(A_\mu \partial_x^3\F + A_\mu \partial_x^3 \G - b_\mu \partial_x \F - b_\mu \partial_x \G\right)\right].
 \end{split}
\end{align}

Eventually, system \eqref{eq:system2} can be written as the following cross-diffusion system
\begin{align}\label{eq:system3}
\begin{split}
\partial_t \F &= -\langle f_0 \rangle \left[ A_\gamma \partial_x^4 \F + A \partial_x^4 \G - b_\rho \partial_x^2 \F - b\partial_x^2 \G \right]+N_{1,A}+N_{1,b}, \\
\partial_t \G&=-\langle g_0 \rangle \left[ A_\mu \partial_x^4 \F + A_\mu \partial_x^4 \G - b_\mu \partial_x^2 \F - b_\mu\partial_x^2 \G \right]+N_{2,A}+N_{2,b},
 \end{split}
\end{align}
where the nonlinear terms $N_{i,A}$ and $N_{i,b}$, $i=1,2$ are given by
\begin{alignat}{2}
\label{N1}
N_{1,A}&:=-\partial_x \left[ \F \left( A_\gamma \partial_x^3 \F + A \partial_x^3 \G \right)\right], \qquad N_{1,b}&&:= \partial_x \left[ \F \left(b_\rho \partial_x \F + b \partial_x \G\right) \right],\\
\label{N3}
N_{2,A}&:=-\partial_x \left[ \G \left( A_\mu \partial_x^3 \F +  A_\mu\partial_x^3 \G \right) \right],\qquad N_{2,b}&&:=  \partial_x \left[ \G \left( b_\mu\partial_x \F +  b_\mu\partial_x \G\right) \right].
\end{alignat}

In the present paper we use the above formulation  \eqref{eq:system3}, i.e. given a positive initial datum $(f_0,g_0)$ for \eqref{eq:system1}, we consider \eqref{eq:system3}, where the initial datum is given by
$$
\F_0 := f_0-\langle f_0 \rangle \qquad \mbox{and} \qquad \G_0 := g_0 - \langle g_0 \rangle,
$$
and the constants $\langle f_0\rangle$ and $\langle g_0\rangle$ are uniquely determined by $f_0$ and $g_0$.

\medskip

\subsection{The thin film Stokes problem}

\subsubsection{The equations}
With the same notation as in the previous section, the two-phase Stokes problem modeling the flow of highly viscous fluids reads
%\begin{subequations}\label{STOKES}
\begin{alignat*}{2}
-\mu_\pm \Delta u_\pm+\nabla p_\pm&=-\rho_\pm G e_2,  \qquad&&\text{in}\quad \Omega_\pm(t)\times[0,T]\,,\\
\nabla\cdot u_\pm &=0,  \qquad&&\text{in}\quad \Omega_\pm(t)\times[0,T]\,,\\
\jump{\mu(\nabla u+(\nabla u)^T)-p\text{Id}}\cdot(-\pax f,1) &= -\gamma_f \mathcal{H}_{\Gamma(f)}(-\pax f,1)\qquad &&\text{on }\Gamma(f)\times[0,T],\\
(\mu_+(\nabla u_++(\nabla u_+)^T)-p_+\text{Id})\cdot(-\pax h,1) &= \gamma_h \mathcal{H}_{\Gamma(h)}(-\pax h,1)\qquad &&\text{on }\Gamma(h)\times[0,T],\\
\jump{u} &= 0\qquad &&\text{on }\Gamma(f)\times[0,T],\\
\pat f &= u_\pm\cdot(-\pax f,1)\qquad &&\text{on }\Gamma(f)\times[0,T],\\
\pat h &= u_+\cdot(-\pax h,1)\qquad &&\text{on }\Gamma(h)\times[0,T],\\
u_- &= 0\qquad &&\text{on }\Gamma_0\times[0,T].
\end{alignat*}
%\end{subequations}

The interaction of two immiscible thin fluid layers with thickness $f$ and $h-f$, respectively (see Figure \ref{fig1}) can be modeled by the two-phase thin film Stokes equation:
\begin{align*} %\label{eq:systemStokes}
\begin{split}
\partial_t f &=\partial_x \left[  2P f^3 \mathscr{D} f + Q(3f^2h-f^3)\mathscr{D} h\right],\\
\partial_t h&=\partial_x \left[ P(3f^2h-f^3) \mathscr{D} f + (2Q\mu(h-f)^3+2Q(h^3-(h-f)^3)) \mathscr{D} h  \right],
\end{split}
\end{align*}
together with \eqref{eq:initial}, \eqref{eq:boundary} and \eqref{positivity}. When only surface tension effects are taken into account, that is $\gamma_h,\gamma_f>0$ and $G=0$, the operator $\mathscr{D}$ and the constants $P$ and $Q$ are given by
\begin{align}\label{eq:systemStokesSurface}
\mathscr{D}=-\pax^3,\qquad P=\frac{\gamma_f}{6\mu_-},\qquad Q=\frac{\gamma_h}{6\mu_-},\qquad \mu=\frac{\mu_-}{\mu_+},
\end{align}
while in the case of a purely gravity driven flow, that is $\gamma_h=\gamma_f=0$ and $G>0$, we have
\begin{align}\label{eq:systemStokesGravity}
\mathscr{D}=\pax,\qquad P=\frac{G(\rho_--\rho_+)}{6\mu_-},\qquad Q=\frac{G\rho_+}{6\mu_-},\qquad \mu=\frac{\mu_-}{\mu_+}.
\end{align}
The above system was derived by Escher, Matioc \& Matioc in \cite{escher2013thin} using lubrication approximation and cross sectional averaging. We remark that for fluids in the Stokes regime gravitational and capillary effects  appear at different order in the approximation. As a consequence, the case where both capillary and gravitational effects are taken into account simultaneously appears to be physically not relevant \cite{escher2013thin}.

\subsubsection{Prior results for the thin film Stokes problem}
There are fewer mathematical results for the thin film Stokes problem. For the gravity driven thin film, Escher, Matioc \& Matioc \cite{escher2013thin} proved local existence of solutions in the Bessel potential space $H^{s,p}$, $2\leq p$ and $s\in(\frac{3}{2},2]$, and exponential convergence towards the flat equilibrium for initial data sufficiently close to their mean in $H^{2,p}$, $p\geq2$.
The results are proved using semigroup theory and the following energy functional
$$
\mathscr{F}_{\text{gravity}}(h,f)=\int |f|^2+\frac{Q}{P}|h|^2dx,
$$
where $Q,P$ are as in \eqref{eq:systemStokesGravity}.
Similarly, when surface tension effects are the only driving force, Escher, Matioc \& Matioc \cite{escher2013thin} proved local existence of solutions in the Bessel potential space $H^{s,p}$, $1< p$ and $s\in(1+\frac{1}{p},4]$, and exponential stability of steady states under a smallness assumption in $H^{4,p}$, $p>1$. As before, the proofs are based on semigroup theory and the energy functional
$$
\mathscr{F}_{\text{capillary}}(h,f)=\int |\pax f|^2+\frac{Q}{P}|\pax h|^2dx,
$$
where $Q,P$ are as in \eqref{eq:systemStokesSurface}.
Eventually, Escher \& Matioc \cite{escher2014non} used the a priori estimates provided by the latter energy functional to prove the existence of nonnegative global weak solutions in the Sobolev space $H^1$.

\subsubsection{Reformulation of the thin film Stokes problem} Following the argument in Section \ref{sectionreformulation}, we consider the equivalent problem with periodic boundary conditions on $[-\pi,\pi]$ for the unknowns $f$ and $g:=h-f$, which reads

\begin{align}\label{eq:systemStokes2}
\begin{split}
\partial_t f &=\partial_x \left[  \left(2(P+Q) f^3+Q3f^2g\right) \mathscr{D} f + Q(3f^2g+2f^3)\mathscr{D}g\right],\\
\partial_t g&=\partial_x \left[ \left(2Q\mu g^3+3(P+Q)f^2g +6Qg^2f)\right) \mathscr{D} f + (2Q\mu g^3+Q(3f^2g+6g^2f)) \mathscr{D} g  \right].
\end{split}
\end{align}

Introducing the new time variable $\tilde t := t/Q$ and suppressing thereafter the tildes, the system above reduces to
\begin{align}
\label{eq:system_Stokesfinal}
\begin{split}
\partial_t f &=\partial_x \left[  \left(2\rho f^3 +3f^2g\right)\mathscr{D} f + (2f^3+3f^2g)\mathscr{D} g\right],\\
\partial_t g&=\partial_x \left[\left(2\mu g^3+3\rho f^2g+6fg^2\right) \mathscr{D} f + \left(2\mu g^3 + 3 f^2g + 6fg^2 \right) \mathscr{D} g  \right],
\end{split}
\end{align}
where
\[
\rho= \frac{P+Q}{Q}.
\]
\begin{lem}[Conservation of mass for \eqref{eq:system_Stokesfinal}]
Let $(f,g)$ be a smooth solution to \eqref{eq:system_Stokesfinal} on $[0,T)$, then the mass of $f$ and $g$ is preserved in time, that is
\[
	\langle f(t) \rangle = \langle f_0 \rangle \quad \mbox{and} \quad	\langle g(t) \rangle = \langle g_0 \rangle \qquad \mbox{for all}\quad t \in [0,T).
\]
\end{lem}

Implementing the zero mean variables \eqref{newvariables}, system \eqref{eq:system_Stokesfinal} can be written as
\begin{align}
\label{eq:system_Stokesfinal2}
\begin{split}
\partial_t \F &= \partial_x \left[\left(2\rho \langle f_0 \rangle^3 +3\langle f_0 \rangle^2\langle g_0 \rangle\right)\mathscr{D}\F + (2\langle f_0 \rangle^3+3\langle f_0 \rangle^2\langle g_0 \rangle)\mathscr{D}\G\right] + N_1 + N_2,\\
\partial_t \G&=\pax \left[\left(2\mu \langle g_0 \rangle^3+3\rho \langle f_0 \rangle^2\langle g_0 \rangle+6\langle f_0 \rangle \langle g_0 \rangle^2\right) \mathscr{D} \F\right] +N_3+N_4\\
&\qquad + \pax\left[\left(2\mu \langle g_0 \rangle^3 + 3 \langle f_0 \rangle^2\langle g_0 \rangle + 6\langle f_0 \rangle \langle g_0 \rangle^2 \right) \mathscr{D} \G\right],
\end{split}
\end{align}
where the nonlinear terms $N_i$, $i=1,\ldots,4$, are given by
\begin{align*}
N_1&:=\partial_x \left[3(\F^2 \G + 2 \F \G \langle f_0 \rangle + \F^2 \langle g_0 \rangle + \G \langle f_0 \rangle^2 +2\F \langle f_0 \rangle \langle g_0 \rangle)\mathscr{D} \F \right]\nonumber\\
&\qquad+\partial_x \left[ 2\rho (\F^3 + 3 \F^2 \langle f_0\rangle + 3 \F \langle f_0 \rangle^2)\mathscr{D} \F \right] %\label{N1Stokes},
\\
N_2&:=\partial_x \left[ 3(\F^2 \G + 2 \F \G \langle f_0 \rangle + \F^2 \langle g_0 \rangle + \G \langle f_0 \rangle^2 +2\F \langle f_0 \rangle \langle g_0 \rangle) \mathscr{D} \G \right]\nonumber\\
&\qquad +\partial_x \left[ 2(\F^3 + 3 \F^2 \langle f_0\rangle + 3 \F \langle f_0 \rangle^2)\mathscr{D} \G \right] %\label{N2Stokes}
,\\
N_3&:=\partial_x \left[\left(2\mu (\G^3 + 3\G^2 \langle g_0 \rangle +3 \G \langle g_0\rangle^2) + 3\rho(\F^2 \G + 2 \F \G \langle f_0 \rangle + \F^2 \langle g_0 \rangle + \G \langle f_0 \rangle^2 +2\F \langle f_0 \rangle \langle g_0 \rangle) \right.\right.\nonumber\\
&\qquad \quad\left.\left.+ 6(\G^2 \F + 2 \G \F \langle g_0 \rangle + \G^2 \langle f_0 \rangle + \F \langle g_0 \rangle^2 +2\G \langle g_0 \rangle \langle f_0 \rangle)\right)\mathscr{D} \F \right] %\label{N3Stokes}
,\\
N_4&:=\partial_x \left[\left(2\mu (\G^3 + 3\G^2 \langle g_0 \rangle +3 \G \langle g_0\rangle^2) + 3(\F^2 \G + 2 \F \G \langle f_0 \rangle + \F^2 \langle g_0 \rangle + \G \langle f_0 \rangle^2 +2\F \langle f_0 \rangle \langle g_0 \rangle) \right.\right.\nonumber\\
&\qquad \quad\left.\left.+ 6(\G^2 \F + 2 \G \F \langle g_0 \rangle + \G^2 \langle f_0 \rangle + \F \langle g_0 \rangle^2 +2\G \langle g_0 \rangle \langle f_0 \rangle)\right)\mathscr{D} \G \right]. %\label{N4Stokes}.
\end{align*}
%\begin{remark}
%Observe that the nonnegativity of the $f$ and $g$ is guaranteed if
%$$
%|\F|\leq\langle f_0\rangle\quad \mbox{and} \quad |\G|\leq\langle g_0\rangle.
%$$
%
%\end{remark}

In the present paper we use the above formulation  \eqref{eq:system_Stokesfinal2}, i.e. given a positive initial datum $(f_0,g_0)$ for \eqref{eq:systemStokes2}, we consider \eqref{eq:system_Stokesfinal2}, where the initial datum is given by
$$
\F_0 := f_0-\langle f_0 \rangle \qquad \mbox{and} \qquad \G_0 := g_0 - \langle g_0 \rangle,
$$
and the constants $\langle f_0\rangle$ and $\langle g_0\rangle$ are uniquely determined by $f_0$ and $g_0$.

\bigskip

\section{Functional framework}\label{ssub:prel}
We write $\TT=[-\pi,\pi]$. Let $n\in \ZZ^+$ and denote by
$$
W^{n,p}(\TT)=\left\{f\in L^p(\TT), \pax^{n} f\in L^p(\TT)\right\}
$$
the standard $L^p$-based Sobolev space with norm
$$
\|f\|_{W^{n,p}}^p=\|f\|_{L^p}^p+\|\pax^n f\|_{L^p}^p.
$$

For a function $u\in L^1(\TT)$ and $k\in \ZZ$ we recall that
$$
\hat{u}(k)=\frac{1}{2\pi}\int_{\TT}u(x)e^{-ix k}dx
$$
denotes the expression of the $k-$th Fourier coefficient of $u$.
If $u\in L^1(\TT)$ and the sequence of its Fourier coefficients $\{\hat u(k)\}_{k\in \ZZ}$ is convergent, then the Fourier series representation of $u$ is given by
$$
u(x)=\sum_{k\in\ZZ}\hat{u}(k)e^{ix k}.
$$

% % % % Facts about Fourier transform % % % %

%Let us collect the following well-known facts:
%\begin{itemize}
%\item If $u\in L_2(\TT)$, then $\|u\|_{L^2}^2=2\pi\sum_{k\in\ZZ}|\hat{u}(k)|^2$.
%\item If $u,v\in L_1(\TT)$, then $\widehat{uv}(k)=\sum_{j\in\ZZ}\hat{u}(j)\hat{v}(k-j)$ for all $k\in \ZZ$.
%\end{itemize}
If $p=2$, we use the notation $H^n(\TT):=W^{n,2}(\TT)$. The $L^2$-based Sobolev spaces on $\TT$ of order $\alpha \in \RR^+$ can be defined by
\[
	H^\alpha(\TT):=\left\{u\in L^2(\TT), \mbox{ such that } \|u\|^2_{H^\alpha}:=\sum_{k\in \ZZ}(1+|k|^{2\alpha})|\hat u(k)|^2 < \infty\right\}.
\]
We use the convention $H^0(\TT)=L^2(\TT)$.
The space consisting of all Lebesgue integrable functions on $\TT$, whose Fourier series is absolutely convergent, is called the \emph{Wiener algebra on $\TT$} and we denote it by $A(\TT)$. In accordance to the definition of Sobolev spaces, we introduce the spaces
\[
	A^\alpha(\TT):=\left\{u\in L^1(\TT), \mbox{ such that } \|u\|_{A^\alpha}:= \sum_{k\in \ZZ}(1+|k|^\alpha)|\hat u(k)|< \infty \right \}
\]
for $\alpha \in \RR^+$ and write $A(\TT)= A^0(\TT)$. Lastly, for $\alpha \in \RR^+$, we denote by $\dot H^\alpha (\TT)$ the space of functions belonging to $H^\alpha(\TT)$ which have zero mean. The space $\dot A^\alpha(\TT)$ is defined accordingly. Notice that
\[
	\|u\|_{\dot H^\alpha}:=\sum_{k\in \ZZ}|k|^{2\alpha}|\hat u(k)|^2\qquad \mbox{and}\qquad \|u\|_{\dot A^\alpha}:=\sum_{k\in \ZZ}|k|^{\alpha}|\hat u(k)|
\]
are equivalent norms on $\dot H^\alpha(\TT)$ and $\dot A^\alpha(\TT)$, respectively.
Moreover, the spaces $\dot{A}^\alpha(\TT)$, $\alpha\in \RR^+$, are Banach algebras and form a Banach scale:
\begin{lem}\label{lem:Wiener} Let $\alpha\in\RR^+$ be a fixed parameter and $f,g\in \dot A^\alpha(\TT)$, then
%$$
%\|fg\|_{\dot{A}^s}\leq 2^{s}\left(\|f\|_{\dot{A}^s(\TT)}\|g\|_{\dot{A}^0(\TT)}+\|f\|_{\dot{A}^0(\TT)}\|g\|_{\dot{A}^s(\TT)}\right)\,,
%$$
\begin{equation}\label{eq:L21}
\|fg\|_{\dot{A}^\alpha}\leq 2^{\alpha+1}\|f\|_{\dot{A}^\alpha(\TT)}\|g\|_{\dot{A}^\alpha(\TT)}.
\end{equation}
Furthermore, the spaces $\dot{A}^\alpha(\TT)$ form a Banach scale with the following interpolation inequality
\begin{equation}\label{interpolation}
\|f\|_{\dot{A}^\alpha}\leq \|f\|_{\dot A}^{1-\theta}\|f\|_{\dot{A}^{\frac{\alpha}{\theta}}}^{\theta}\quad \mbox{for all }\quad 0<\theta<1.
\end{equation}
\end{lem}
\begin{proof}
Let $\alpha \in \RR^+$ be fixed and $f,g\in \dot{A}^\alpha(\TT)$. The product of $f$ and $g$ can be represented as
\[
fg(x)= \sum_{n\in \ZZ}\left (\sum_{m\in \ZZ}\hat f(n-m)\hat g(m)\right)e^{-inx}.
\]
Due to the basic inequality
$$
|n|^\alpha\leq 2^\alpha\max\{|n-m|,|m|\}^\alpha\leq 2^\alpha\left(|n-m|^\alpha+|m|^\alpha\right)
$$
we obtain the estimate
\begin{align*}
\|fg\|_{\dot{A}^\alpha}&=\sum_{n\in \ZZ}|n|^\alpha\left|\sum_{m\in \ZZ}\hat f(n-m)\hat g(m)\right|\\
&\leq \sum_{n\in \ZZ}\sum_{m\in \ZZ}|n|^\alpha|\hat f(n-m)| |\hat g(m)|\\
&\leq 2^{\alpha}\sum_{n\in \ZZ}\sum_{m\in \ZZ} \left(|n-m|^\alpha+|m|^\alpha\right)|\hat f(n-m)| |\hat g(m)|\\
&\leq 2^{\alpha}\left(\|f\|_{\dot{A}^\alpha}\|g\|_{\dot{A}}+\|f\|_{\dot{A}}\|g\|_{\dot{A}^\alpha}\right)\\
&\leq 2^{\alpha+1} \|f\|_{\dot{A}^\alpha}\|g\|_{\dot{A}^\alpha},
\end{align*}
which proves inequality \eqref{eq:L21}.
The interpolation inequality \eqref{interpolation} is due to the H\"older inequality for $p=1/\theta$ and $q=1/(1-\theta)$:
\begin{align*}
\|f\|_{\dot{A}^\alpha}&=\sum_{n\in\ZZ}|n|^\alpha|\hat{f}(n)|^{\theta}|\hat{f}(n)|^{1-\theta}\leq \left(\sum_{n\in\ZZ}|n|^{\frac{\alpha}{\theta}}|\hat{f}(n)|\right)^{\theta}\left(\sum_{n\in\ZZ}|\hat{f}(n)|\right)^{1-\theta}= \|f\|_{\dot A}^{1-\theta}\|f\|_{\dot{A}^{\frac{\alpha}{\theta}}}^{\theta}.
\end{align*}
\end{proof}

Let us collect some embedding properties of the spaces $ A^\alpha(\TT)$. Clearly, for any $\alpha, \beta \in \RR^+$ with $\alpha \geq \beta$ we have that
$ A^\alpha(\TT) \subset A^\beta (\TT)$. Moreover, it is easy to verify that 
\[
	C^{k+1}(\TT) \subset  A^k (\TT) \subset C^{k}(\TT) \qquad \mbox{for all}\quad k\in \NN.
\]
In addition we introduce the space  $\mathcal{M}(0,T;X)$, the space of Radon measures from an interval $[0,T]$ to a Banach space $X$.

\medskip

Eventually, we end this section by a comment on a general convention: We denote by $c>0$ a generic constant, which may differ from occurrence to occurrence. Sometimes we use the notation $c=c(\cdot, \cdot, \ldots)$ in order to emphasize the dependence of $c$ on various parameters.

%The fractional $L^p$-based Sobolev spaces, $W^{s,p}(\TT)$, are defined as
%$$
%W^{s,p}(\TT)=\left\{f\in L^p(\TT), \pax^{\lfloor s\rfloor} f\in L^p, \frac{|\pax^{\lfloor s\rfloor}f(x)-\pax^{\lfloor s\rfloor}f(y)|}{|x-y|^{\frac{1}{p}+(s-\lfloor s\rfloor)}}\in L^p(\TT^2)\right\},
%$$
%with norm
%$$
%\|f\|_{W^{s,p}}^p=\|f\|_{L^p}^p+\|f\|_{\dot{W}^{s,p}}^p, 
%$$
%$$
%\|f\|_{\dot{W}^{s,p}}^p=\|\pax^{\lfloor s\rfloor} f\|^p_{L^p}+\int_\TT\int_\TT\frac{|\pax^{\lfloor s\rfloor}f(x)-\pax^{\lfloor s\rfloor}f(y)|^p}{|x-y|^{1+(s-\lfloor s \rfloor)p}}.
%$$
%Notice that $W^{k+s,\infty}(\TT)$, $k\in\NN$, $1>s\geq0$, reduces to the usual space $C^{k+s}(\TT)$ of H\"older continuous functions.

\bigskip

\section{Main results and discussion}
The goal of this paper is to obtain the global existence and decay towards equilibria for the thin film Muskat and the thin film Stokes problems for appropriate initial data. In particular, our results consider both the gravity driven case (when surface tension effects are neglected, i.e. $\gamma_h=\gamma_f=0$) and the capillary driven case (when $\gamma_h,\gamma_f\neq0$). 

\subsection{The thin film Muskat problem}
First we introduce our notion of weak solution for \eqref{eq:system3} when surface tension effects are considered:
\begin{definition}\label{defi1}We say that a pair of zero mean functions $(\F,\G)\in \left(L^1\left(0,T;W^{3,1}(\TT)\right)\right)^2$ is a weak solution of \eqref{eq:system3} corresponding to the initial datum $(\F_0,\G_0)$ if and only if
\begin{multline*}
-\int_\TT \F_0\phi(0)dx-\int_0^T\int_\TT \F\pat\phi dxdt=\langle f_0 \rangle\int_0^T\int_\TT\pax^3\left[ A_\gamma \F + A \G\right]\pax\phi + \pax\left(b_\rho \F + b \G \right)\pax\phi dxdt\\
+\int_0^T\int_\TT\left(- \F \left( A_\gamma \partial_x^3 \F + A \partial_x^3 \G \right)+ \F \left(b_\rho \partial_x \F + b \partial_x \G\right) \right)\pax\phi dxdt, 
\end{multline*}
and
\begin{multline*}
-\int_\TT \G_0\psi(0)dx-\int_0^T\int_\TT \G\pat\psi dxdt=\langle g_0 \rangle \int_0^T\int_\TT A_\mu\pax^3\left[  \F +  \G\right]\pax\psi+b_\mu\pax\left[  \F +  \G\right]\pax\psi dxdt \\
+\int_0^T\int_\TT\left[-\G \left( A_\mu \partial_x^3 \F +  A_\mu\partial_x^3 \G \right) +\F \left( b_\mu\partial_x \F +  b_\mu\partial_x \G\right) \right]\pax\psi dxdt,
\end{multline*}
for all $(\phi,\psi)\in C^{1}_c([0,T)\times \TT)$.
\end{definition}

When surface tension is neglected, our definition of weak solutions for \eqref{eq:system3} reads:
\begin{definition}\label{defi2}We say that a pair of zero mean functions $(\F,\G)\in \left(L^1\left(0,T;W^{1,1}(\TT)\right)\right)^2$ is a weak solution of \eqref{eq:system3} corresponding to the initial datum $(\F_0,\G_0)$ if and only if
\begin{align*}
-\int_\TT \F_0\phi(0)dx-\int_0^T\int_\TT \F\pat\phi dxdt=&\langle f_0 \rangle\int_0^T\int_\TT\pax\left(b_\rho \F + b \G \right)\pax\phi dxdt\\
&+\int_0^T\int_\TT \F \left(b_\rho \partial_x \F + b \partial_x \G\right) \pax\phi dxdt, 
\end{align*}
and
\begin{align*}
-\int_\TT \G_0\psi(0)dx-\int_0^T\int_\TT \G\pat\psi dxdt=&\langle g_0 \rangle \int_0^T\int_\TT b_\mu\pax\left[  \F +  \G\right]\pax\psi dxdt \\
&+\int_0^T\int_\TT\F \left( b_\mu\partial_x \F +  b_\mu\partial_x \G\right)\pax\psi dxdt,
\end{align*}
for all $(\phi,\psi)\in C^{1}_c([0,T)\times \TT)$.
\end{definition}

Before stating the main results, some notation needs to be introduced. We define the following functionals:
\begin{align}
\mathcal{E}_s(\F,\G)&:=  \|\F\|_{\dot{A}^s} + \|\G\|_{\dot{A}^s},\label{notation1}\\
E_s(\F,\G)&:=  \|\F\|_{\dot{H}^s}^2 + \|\G\|_{\dot{H}^s}^2,\label{notation1b}\\
\mathscr{E}_n(\F,\G)&:=\|\pax^n\F\|_{L^\infty} + \|\pax^n\G\|_{L^\infty}.\label{notation1c}
\end{align}
Moreover, we set
%\begin{subequations}\label{notation2}
\begin{align*}
\sigma_{1,A}&:= \langle f_0 \rangle A_\gamma- \langle g_0 \rangle A_\mu - (A_\mu+2A_\gamma+2A)\mathcal{E}_0(\bar{f}_0,\bar{g}_0),\\
\sigma_{2,A}&:=\langle g_0 \rangle A_\mu - \langle f_0 \rangle A - (A_\mu+2A_\gamma+2A)\mathcal{E}_0(\bar{f}_0,\bar{g}_0),\\
\sigma_{1,b}&:= \langle f_0 \rangle b_\rho - \langle g_0 \rangle b_\mu - \mathcal{E}_0(\bar{f}_0,\bar{g}_0)(2b_\rho+2b+4b_\mu),\\
\sigma_{2,b}&:=\langle g_0 \rangle b_\mu - \langle f_0 \rangle b - \mathcal{E}_0(\bar{f}_0,\bar{g}_0)(2b_\rho+2b+4b_\mu).
\end{align*}
%\end{subequations}
Then, the first result is formulated as follows.
\begin{theorem}[Two-phase thin film Muskat system with surface tension]\label{theorem1} Let $\gamma_f,\gamma_h>0$ and $(\bar{f}_0, \bar{g}_0)\in \left(\dot A(\TT)\right)^2$ be the initial datum for \eqref{eq:system3} satisfying
\[
\mathcal{E}_0(\bar{f}_0, \bar{g}_0)<\min\{\langle f_0\rangle,\langle g_0\rangle\},
\qquad \min\{\sigma_{1,A},\sigma_{2,A},\sigma_{1,b},\sigma_{2,b}\}>0.
\]
Then:
\begin{itemize}
\item[a)] \textbf{Existence:} There exist at least one global weak solution in the sense of Definition \ref{defi1} of \eqref{eq:system3} having the regularity
\begin{align*}
(\F,\G)\in &\Big(L^\infty\left([0,T]\times \TT\right)\cap L^{\frac{4}{3}}\left(0,T;W^{3,\infty}(\TT)\right)\cap L^1\left(0,T;C^{3+\alpha}(\TT)\right)  \\
& \qquad \cap \mathcal{M}\left(0,T;W^{4,\infty}(\TT)\right)\cap L^2\left(0,T;\dot H^2(\TT)\right) \Big)^2
\end{align*}
for any $T>0$, where $\alpha \in [0,\frac{1}{2})$.
\item[b)] \textbf{Exponential decay:} The solution satisfies
$$
\|\F(T)\|_{L^\infty}+\|\G(T)\|_{L^\infty}\leq 	\mathcal{E}_0(\F_0,\G_0)e^{-(\delta_A + \delta_b) T},
$$
where $\delta_A=\delta_A(\F_0,\G_0, \rho_{\pm}, \mu_{\pm}, \gamma_f, \gamma_h)>0$ and $\delta_b=\delta_b(\F_0,\G_0, \rho_{\pm}, \mu_{\pm})>0$ are certain explicit constants depending on the initial datum and the physical parameters.
\item[c)] \textbf{Uniqueness:}
If 
$$
(\F,\G)\in \left(L^1\left(0,T;\dot{A}^4(\TT)\right)\right)^2,
$$
then the weak solution is unique.
\end{itemize}
\end{theorem}

\begin{remark}
Theorem \ref{theorem1} concerns initial data satisfying a size restriction in the Wiener Algebra $A(\TT)$. Since the restriction is explicit (and $O(1)$) in terms of the parameters of the problem, we say that the initial data are of \emph{medium size}. In particular, the initial datum can be arbitrary large in $H^s(\TT)$, $s>0$.
\end{remark}

\begin{remark}
Notice that a necessary condition for $\sigma_{1,b},$ and $\sigma_{2,b}$ to be positive is that
\[
	\rho_->\rho_+.
\]
Thus, the fluid with higher density is below, which is a reasonable assumption for a gravity driven flow. If capillary forces are included, then $\sigma_{1,A}$ and $\sigma_{2,A}$ can only be positive if both $\gamma_h$ and $\gamma_f$ are strictly positive.
\end{remark}

In the case of a purely gravity driven flow (when surface tension effects are neglected), we can formulate a similar theorem as above.
\begin{theorem}[Two-phase thin film Muskat system without surface tension]\label{theorem2}
Let $\gamma_f,\gamma_h=0$ and $(\bar{f}_0, \bar{g}_0)\in \left(\dot A(\TT)\right)^2$ be the initial datum for \eqref{eq:system3} satisfying
\[
\mathcal{E}_0(\bar{f}_0, \bar{g}_0)<\min\{\langle f_0\rangle,\langle g_0\rangle\},
\qquad \min\{\sigma_{1,b},\sigma_{2,b}\}>0.
\]
Then:
\begin{itemize}
\item[a)] \textbf{Existence:} There exist at least one global weak solution in the sense of Definition \ref{defi1} of \eqref{eq:system3} having the regularity
\begin{align*}
(\F,\G)\in &\Big(L^\infty\left([0,T]\times \TT)\right)\cap L^{2}\left(0,T;W^{1,\infty}(\TT)\right)\cap L^1\left(0,T;C^{1+\alpha}(\TT)\right) \\
&\qquad  \cap \mathcal{M}\left(0,T;W^{2,\infty}(\TT)\right)\cap L^2\left(0,T;\dot H^1(\TT)\right) \Big)^2
\end{align*}
for any $T>0$, where $\alpha \in [0,\frac{1}{2})$.
\item[b)] \textbf{Exponential decay:} The solution satisfies
$$
\|\F(T)\|_{L^\infty}+\|\G(T)\|_{L^\infty}\leq 	\mathcal{E}_0(\F(_0,\G_0)e^{-\delta_b T}.
$$
where $\delta_b=\delta_b(\F_0,\G_0, \rho_{\pm}, \mu_{\pm})>0$ is a certain explicit constants depending on the initial datum and the physical parameters.
\item[c)] \textbf{Uniqueness:}  If 
$$
(\F,\G)\in \left(L^1\left(0,T;\dot{A}^2(\TT)\right)\right)^2,
$$
then the weak solution is unique.
\end{itemize}
%there exist at least one global weak solution in the sense of Definition \ref{defi1} of \eqref{eq:system3} having the regularity
%\begin{align*}
%(\F,\G)\in &\Big(L^\infty([0,T]\times \TT)\cap L^{2}(0,T;W^{1,\infty}(\TT))\cap L^1(0,T;C^{1+\alpha}(\TT)) \cap \mathcal{M}(0,T;W^{2,\infty}(\TT)) \Big)^2
%\end{align*}
%for any $T>0$, where $\alpha \in [0,\frac{1}{2})$.
%Moreover, the solution satisfies
%$$
%\|\F(T)\|_{L^\infty}+\|\G(T)\|_{L^\infty}\leq 	\mathcal{E}_0(\F(0),\G(0))e^{-C T}.
%$$
%for a certain explicit constant $C=C(\F_0,\G_0, \rho_{\pm}, \mu_{\pm})>0$.
%Furthermore, if 
%$$
%(\F,\G)\in \left(L^1(0,T;\dot{A}^2(\TT))\right)^2,
%$$
%then, the weak solution is unique.
\end{theorem}

\begin{remark}
Related results for the free boundary Muskat problem without the thin film assumption can be found in \cite{ constantin2016muskat, ccgs-10, gancedo2017muskat}.
\end{remark}

If we increase the regularity of the initial data and assume some additional restrictions on $\mathcal{E}_0(\F_0,\G_0)$ we can propagate Sobolev regularity of the solution:
\begin{theorem}[Two-phase thin film Muskat with surface tension -- Sobolev regularity]\label{theorem3}
Let $\gamma_f,\gamma_h>0$  and $(\bar{f}_0, \bar{g}_0)\in \left( \dot{H}^2(\TT)\right)^2$ be the initial datum for \eqref{eq:system3}  satisfying
\[
\mathcal{E}_0(\bar{f}_0, \bar{g}_0)<\min\{\langle f_0\rangle,\langle g_0\rangle\},
\qquad \min\{\sigma_{1,A},\sigma_{2,A},\sigma_{1,b},\sigma_{2,b}\}>0.
\]
If in addition
\begin{align*} %\label{eq:extra0}
\begin{split}
\langle g_0 \rangle A_\mu-\frac{\langle f_0 \rangle A+\langle g_0 \rangle A_\mu}{2}-\left(A_\gamma +\frac{13}{4}A+\frac{17}{4}A_\mu\right)\mathcal{E}_0(\F,\G)>0\\
\langle f_0 \rangle A_\gamma-\frac{\langle f_0 \rangle A+\langle g_0 \rangle A_\mu}{2}-\left(A_\gamma +\frac{13}{4}A+\frac{17}{4}A_\mu\right)\mathcal{E}_0(\F_0,\G_0)>0,
\end{split}
\end{align*}
then the global weak solution of \eqref{eq:system3} obtained in Theorem \ref{theorem1} also satisfies for all $T>0$ that
$$
(\F,\G)\in \left(C\left(0,T;\dot{H}^2(\TT)\right)\cap L^2\left(0,T;\dot{H}^4(\TT)\right)\right)^2
$$
with 
\begin{itemize}\setlength\parskip{10pt}
\item[i)]
$
E_2(\F(T),\G(T))+c_1\int_0^T E_4(\F(s),\G(s))ds\leq c_2,
$
\item[ii)] $
E_s(\F(T),\G(T))\leq c_3e^{-c T} \quad \mbox{for all} \quad 0\leq s<2,
$
%\item[iii)] 
%$
%\|\F(T)\|_{\dot{W}^{1,4}}^2+\|\G(T)\|_{\dot{W}^{1,4}}^2\leq C_4e^{-C_5 T},
%$
\end{itemize}
for certain positive constants $c=c(\F_0,\G_0,\rho_{\pm}, \mu_{\pm}, \gamma_f, \gamma_h,s)$, $c_i=c_i(\F_0,\G_0,\rho_{\pm}, \mu_{\pm}, \gamma_f, \gamma_h)$, $i=1,2,3$.
\end{theorem}

Let us remark that in Theorem \ref{theorem3} there are no size restrictions on the initial datum $(\F_0,\G_0)$ in $H^2(\TT)$.
%Moreover, our results says nothing on the convergence towards $0$ in the $H^2$ norm of the solution $(\F(t),\G(t))$.

\begin{remark}
For the hypotheses of Theorem \ref{theorem3} to fulfill, the physical parameters and the initial datum $(f_0,g_0)$ have to satisfy
$$
\langle g_0 \rangle A_\mu-\langle f_0 \rangle A>0,
$$
$$
\langle f_0 \rangle A_\gamma-\langle g_0 \rangle A_\mu>0.
$$ 
If both fluids have the same viscosity, the above condition requires that
\[
 \frac{\gamma_f+\gamma_h}{\gamma_h}\langle f_0 \rangle>	\langle g_0 \rangle > \langle f_0 \rangle. 
\]
\end{remark}

Analogously to Theorem \ref{theorem3}, the following result for the gravity driven two-phase thin film Muskat holds true:
\begin{theorem}[Two-phase thin film Muskat without surface tension -- Sobolev regularity]\label{theorem4}
Let $\gamma_f,\gamma_h=0$  and $(\bar{f}_0, \bar{g}_0)\in \left(\dot{H}^1(\TT)\right)^2$ be the initial datum for \eqref{eq:system3} such that
\[
\mathcal{E}_0(\bar{f}_0, \bar{g}_0)<\min\{\langle f_0\rangle,\langle g_0\rangle\},
\qquad \min\{\sigma_{1,b},\sigma_{2,b}\}>0.
\]
If in addition
\begin{align*} %\label{extra}
\begin{split}
\langle f_0 \rangle b_\rho-\frac{1}{2}\left(\langle g_0\rangle b_\mu+\langle f_0 \rangle b\right)-\left(b_\rho+b_\mu+\frac{5b_\mu}{2}+\frac{5b}{2}\right)\mathcal{E}_0(\F_0,\G_0)&>0,\\
\langle g_0\rangle b_\mu-\frac{1}{2}\left(\langle g_0\rangle b_\mu+\langle f_0 \rangle b\right)-\left(b_\rho+b_\mu+\frac{5b_\mu}{2}+\frac{5b}{2}\right)\mathcal{E}_0(\F_0,\G_0)&>0,
\end{split}
\end{align*}

then the global weak solution of \eqref{eq:system3} obtained in Theorem \ref{theorem2} also satisfies for all $T>0$ that
$$
(\F,\G)\in \left(C\left(0,T;\dot{H}^1(\TT)\right)\cap L^2\left(0,T;\dot{H}^2(\TT)\right)\right)^2
$$
with 
\begin{itemize}\setlength\parskip{10pt}
\item[i)]
$
E_1(\F(T),\G(T))+c_1\int_0^T E_2(\F(s),\G(s))ds\leq c_2,
$ 
\item[ii)] $
E_s(\F(T),\G(T))\leq c_3e^{-c T} \quad \mbox{for all} \quad 0\leq s<1,
$
\end{itemize}
for certain positive constants $c=c(\F_0,\G_0,\rho_{\pm}, \mu_{\pm},s)$, $c_i=c_i(\F_0,\G_0,\rho_{\pm}, \mu_{\pm})$, $i=1,2,3$.
\end{theorem}

\begin{remark}
Theorems \ref{theorem1}, \ref{theorem2}, \ref{theorem3} and \ref{theorem4} can also be stated in terms of the solutions to \eqref{eq:system1}. On the one hand, Theorem \ref{theorem1} and Theorem \ref{theorem2} show the global existence of positive solutions for $(f_0,g_0)\in \left( A(\TT)\right)^2$ of \eqref{eq:system1} and its uniform convergence towards $(\langle f_0\rangle,\langle g_0\rangle)$. On the other hand, Theorems \ref{theorem3} and \ref{theorem4} prove that the solution propagates Sobolev regularity if further (explicit) smallness conditions on the (weak) norm  of the initial data in $ A(\TT)$ are assumed.  
\end{remark}

\subsection{The thin film Stokes problem}
Our definition of a weak solution for the thin film Stokes problem \eqref{eq:system_Stokesfinal2} is given by:
\begin{definition}\label{defi3}Define $\zeta=1$ if $\mathscr{D}=\pax$ (gravity driven flow) and $\zeta=3$ if $\mathscr{D}=-\pax^3$ (capillary driven flow). We say that $(\F,\G)\in L^1\left(0,T;W^{\zeta,1}(\TT)\right)$ is a weak solution to \eqref{eq:system_Stokesfinal2} corresponding to the initial datum $(\F_0,\G_0)$ if and only if
\begin{align*}
\int_\TT &\F_0\phi(0) dx+\int_0^T\int_\TT \F\partial_t\phi dtdx =\int_0^T\int_\TT\bigg{[}\left(2\rho \langle f_0 \rangle^3 +3\langle f_0 \rangle^2\langle g_0 \rangle\right)\mathscr{D}\F\bigg{]}\pax\phi dxdt\\
&+ \int_0^T\int_\TT\Big{[}(2\langle f_0 \rangle^3+3\langle f_0 \rangle^2\langle g_0 \rangle)\mathscr{D}\G\Big{]}\pax\phi dxdt\\
&+ \int_0^T\int_\TT\Big[3(\F^2 \G + 2 \F \G \langle f_0 \rangle + \F^2 \langle g_0 \rangle + \G \langle f_0 \rangle^2 +2\F \langle f_0 \rangle \langle g_0 \rangle) \\
&\qquad\qquad +  2\rho (\F^3 + 3 \F^2 \langle f_0\rangle + 3 \F \langle f_0 \rangle^2))\mathscr{D} \F \Big]\partial_x \phi\, dxdt\\
&+\int_0^T\int_\TT \left[ 3(\F^2 \G + 2 \F \G \langle f_0 \rangle + \F^2 \langle g_0 \rangle + \G \langle f_0 \rangle^2 +2\F \langle f_0 \rangle \langle g_0 \rangle)\mathscr{D} \F \right]\pax\phi dxdt, 
\end{align*}
and
\begin{align*}
\int_\TT &\G_0\psi(0) dx+\int_0^T\int_\TT \G\partial_t\psi dtdx =\int_0^T\int_\TT\Big{[}\left(2\mu \langle g_0 \rangle^3+3\rho \langle f_0 \rangle^2\langle g_0 \rangle+6\langle f_0 \rangle \langle g_0 \rangle^2\right) \mathscr{D} \F\Big{]}\pax\psi dxdt\\
&+ \int_0^T\int_\TT\Big{[}\left(2\mu \langle g_0 \rangle^3 + 3 \langle f_0 \rangle^2\langle g_0 \rangle + 6\langle f_0 \rangle \langle g_0 \rangle^2 \right) \mathscr{D} \G\Big{]}\pax\phi dxdt\\
&+ \int_0^T\int_\TT\bigg{[}\Big{(}2\mu (\G^3 + 3\G^2 \langle g_0 \rangle +3 \G \langle g_0\rangle^2) + 3\rho(\F^2 \G + 2 \F \G \langle f_0 \rangle + \F^2 \langle g_0 \rangle + \G \langle f_0 \rangle^2 +2\F \langle f_0 \rangle \langle g_0 \rangle) \\
&\qquad \qquad + 6(\G^2 \F + 2 \G \F \langle g_0 \rangle + \G^2 \langle f_0 \rangle + \F \langle g_0 \rangle^2 +2\G \langle g_0 \rangle \langle f_0 \rangle)\bigg{)}\mathscr{D} \F \Big{]}\pax \psi dxdt\\
&+\int_0^T\int_\TT \Big{[}\Big{(}2\mu (\G^3 + 3\G^2 \langle g_0 \rangle +3 \G \langle g_0\rangle^2) + 3(\F^2 \G + 2 \F \G \langle f_0 \rangle + \F^2 \langle g_0 \rangle + \G \langle f_0 \rangle^2 +2\F \langle f_0 \rangle \langle g_0 \rangle) \\
&\qquad \qquad+ 6(\G^2 \F + 2 \G \F \langle g_0 \rangle + \G^2 \langle f_0 \rangle + \F \langle g_0 \rangle^2 +2\G \langle g_0 \rangle \langle f_0 \rangle)\Big{)}\mathscr{D} \G \Big{]}\pax\phi dxdt, 
\end{align*} for all $(\phi,\psi)\in C^{\infty}_c([0,T)\times \TT)$.
\end{definition}

Let us define the following constants:

\begin{align*} %\label{SIGMA0}
\begin{split}
\Sigma_1&=2\rho \langle f_0 \rangle^3 +3\langle f_0 \rangle^2\langle g_0 \rangle(1-\rho)-\left(2\mu \langle g_0 \rangle^3+6\langle f_0 \rangle \langle g_0 \rangle^2\right)\\
&-\mathcal{E}_0(\F_0,\G_0)\bigg{[}({78}+20\rho+14\mu)\mathcal{E}_0(\F_0,\G_0)^2+(\langle f_0 \rangle(36\rho+84)+({81}+30\mu+{9}\rho)\langle g_0 \rangle)\mathcal{E}_0(\F_0,\G_0)\bigg{]}\\
&-\mathcal{E}_0(\F_0,\G_0)\bigg{[}(18\mu+{18}) \langle g_0\rangle^2 + ({18\rho}+ {18})\langle f_0 \rangle^2 +({12\rho}+{60}) \langle g_0 \rangle \langle f_0 \rangle\bigg{]},
\end{split}
\end{align*}
\begin{align*} %\label{SIGMA0b}
\begin{split}
\Sigma_2&=2\mu \langle g_0 \rangle^3  + 6\langle f_0 \rangle \langle g_0 \rangle^2-2\langle f_0 \rangle^3 \\
&-\mathcal{E}_0(\F_0,\G_0)\bigg{[}(14\mu+15\rho+83)\mathcal{E}_0(\F_0,\G_0)^2+\left((30\mu+6\rho+{84})\langle g_0 \rangle+(96+24\rho)\langle f_0 \rangle\right)\mathcal{E}_0(\F_0,\G_0)\bigg{]}\\
&-\mathcal{E}_0(\F_0,\G_0)\bigg{[}(18\mu+18) \langle g_0\rangle^2 + (27+9\rho)\langle f_0 \rangle^2 +({6\rho}+{66}) \langle f_0 \rangle \langle g_0 \rangle\bigg{]}\,.
\end{split}
\end{align*}

Then, our main result for the thin film Stokes equations reads as follows.
\begin{theorem}[Two-phase thin film Stokes system]\label{theorem5} Let $(\bar{f}_0, \bar{g}_0)\in \left(\dot A(\TT)\right)^2$ be the initial datum for \eqref{eq:system_Stokesfinal2} such that
$$
\mathcal{E}_0(\bar{f}_0, \bar{g}_0)<\min\{\langle f_0\rangle,\langle g_0\rangle\}.
$$
Define $\zeta=1$ if $\mathscr{D}=\pax$ (gravity driven flow) and $\zeta=3$ if $\mathscr{D}=-\pax^3$ (capillary driven flow). Assume that 
$$
\min\{\Sigma_1,\Sigma_2\}>0.
$$
Then:
\begin{itemize}
\item[a)] \textbf{Existence:} There exist at least one global weak solution in the sense of Definition \ref{defi3} of \eqref{eq:system_Stokesfinal2} having the regularity
\begin{align*}
(\F,\G)\in &\Big(L^\infty\left([0,T]\times \TT\right)\cap L^{\frac{\zeta+1}{\zeta}}\left(0,T;W^{\zeta,\infty}(\TT)\right)\cap L^1\left(0,T;C^{{\zeta}+\alpha}(\TT)\right) \\
&\qquad \mathcal{M}\left(0,T;W^{\zeta+1,\infty}(\TT)\right)\cap L^2\left(0,T;\dot{H}^{(\zeta+1)/2}(\TT)\right)\Big)^2
\end{align*}
for any $T>0$, where $\alpha\in [0,\frac{1}{2})$.
\item[b)] \textbf{Exponential decay:} The solution satisfies
$$
\|\F(T)\|_{L^\infty}+\|\G(T)\|_{L^\infty}\leq 	\mathcal{E}_0(\F_0,\G_0)e^{-\varepsilon T},
$$
where $\varepsilon=\varepsilon(\F_0,\G_0, \rho_{\pm}, \mu_{\pm},\zeta)>0$ is a certain explicit constant depending on the initial datum and the physical parameters.
\item[c)] \textbf{Uniqueness:}
If 
$$
(\bar{f}_0, \bar{g}_0)\in \left(L^1\left(0,T;\dot{A}^{\zeta+1}(\TT)\right)\right)^2,
$$
then, the weak solution is unique.
\end{itemize}
\end{theorem}
We define the following constants:
\begin{align*}
\mathscr{C}_1&:=\frac{1}{2}{\mathcal{E}}_0(\F_0,\G_0)^2\left(23+5\rho+4\mu \right)+\frac{1}{2}{\mathcal{E}}_0(\F_0,\G_0)\left( \langle f_0 \rangle (36+12\rho)+ \langle g_0 \rangle (33+3\rho+4\mu)\right)\\
&\qquad  +\frac{1}{2}\left( \langle f_0 \rangle^2(15+9\rho)+ \langle g_0 \rangle^2(12+12\mu)+ \langle f_0 \rangle \langle g_0 \rangle (42+6\rho) \right),\\
\mathscr{C}_{3}&=\mathcal{E}_0(\F_0,\G_0)^2\left (32+\frac{55}{8}\rho +\frac{15}{2}\mu \right)+\mathcal{E}_0(\F_0,\G_0) \left(\langle f_0 \rangle \left(66+18 \rho \right) + \langle g_0 \rangle \left(36+\frac{21}{2}\rho+18\mu \right)\right)\\
&\qquad +\left( \langle f_0 \rangle^2 \left(\frac{57}{4}+\frac{39}{4}\rho \right) +\langle g_0 \rangle^2 \left(\frac{27}{2}+\frac{15}{2}\mu \right) + \langle f_0 \rangle\langle g_0 \rangle \left( \frac{81}{2}+\frac{15}{2}\rho\right)   \right).
\end{align*}

\begin{theorem}[Two-phase thin film Stokes system -- Sobolev regularity]\label{theorem6}
Let the initial datum $(\bar{f}_0, \bar{g}_0)\in \left(\dot H^{(\zeta+1)/2}(\TT)\right)^2$  for \eqref{eq:system_Stokesfinal2} be such that
\[
\mathcal{E}_0(\bar{f}_0, \bar{g}_0)<\min\{\langle f_0\rangle,\langle g_0\rangle\}.
\]
Define $\zeta=1$ if $\mathscr{D}=\pax$ (gravity driven flow) and $\zeta=3$ if $\mathscr{D}=-\pax^3$ (capillary driven flow). Assume that 
$$
\min\{\Sigma_1,\Sigma_2\}>0 
$$
and 
\begin{align*}
(2\rho-1) \langle f_0 \rangle^3 - \frac{3}{2}(\rho-1) \langle f_0 \rangle^2 \langle g_0 \rangle - 3\langle f_0 \rangle \langle g_0 \rangle^2-\mu \langle g_0 \rangle^3 -{\mathcal{E}}_0(\F_0,\G_0)\mathscr{C}_\zeta &>0,\\
\mu \langle g_0 \rangle^3+3\langle f_0 \rangle \langle g_0 \rangle^2 -\frac{3}{2}(\rho-1)\langle f_0 \rangle^2 \langle g_0 \rangle - \langle f_0\rangle^3 -{\mathcal{E}}_0(\F_0,\G_0)\mathscr{C}_\zeta&>0.
\end{align*}
Then, the global weak solution of \eqref{eq:system_Stokesfinal2} obtained in Theorem \ref{theorem5} also satisfies for all $T>0$ that
$$
(\F,\G)\in \left(C\left(0,T;\dot{H}^{(\zeta+1)/2}(\TT)\right)\cap L^2\left(0,T;\dot{H}^{\zeta+1}(\TT)\right)\right)^2
$$
with 
\begin{itemize}\setlength\parskip{10pt}
\item[i)]
$
E_{(\zeta+1)/2}(\F(T),\G(T))+c_1\int_0^T E_{\zeta+1}(\F(s),\G(s))ds\leq c_2,
$
\item[ii)] $
E_s(\F(t),\G(t))\leq c_3e^{-c T}$\quad for all \quad $0\leq s<(\zeta+1)/2,$
\end{itemize}
for certain positive constants $c=c(\F_0,\G_0,\mu,\rho,\zeta,s)$, $c_i=c_i(\F_0,\G_0,\mu,\rho,\zeta)$, $i=1,2,3$.
\end{theorem}

\medskip

The remaining of the present work is devoted to the proofs of Theorem \ref{theorem1} - Theorem \ref{theorem6} given in Section \ref{section4} - Section \ref{section9}. The main ideas can be found in Section \ref{section4} (proof of Theorem \ref{theorem1}), where we show the existence of global weak solutions for initial data in the Wiener algebra $A(\TT)$ with explicit decay rates towards equilibia, and in Section \ref{section6} (proof of Theorem \ref{theorem3}), where we show that if the initial data satisfy additionally Sobolev regularity and some size restrictions, we can propagate Sobolev regularity for the corresponding global weak solution.

\bigskip

\section{Existence and decay for the capillary driven thin film Muskat system in the Wiener algebra}\label{section4}

The proof of Theorem \ref{theorem1} is split into several steps. Let us first observe that any local solution of \eqref{eq:system3} satisfies some a priori energy estimates in the Wiener spaces.

\begin{lem}[Energy estimate]\label{lem:EE} If $(\F,\G)\in \left(C^1\left([0,T); \dot A(\TT)\right)\right)^2$ is a local solution of \eqref{eq:system3} to the initial datum $(\F_0, \G_0)\in \left( \dot A(\TT)\right)^2$ satisfying the size restriction
\begin{equation}\label{eq:size}
	\sigma_{1,A}, \sigma_{2,A}, \sigma_{1,b}, \sigma_{2,b}>0,
\end{equation}
then there exist $\delta_A, \delta_b>0$ such that
\[
\frac{d}{dt}\mathcal{E}_0(\F(t), \G(t))+\delta_A\mathcal{E}_{4}(\F(t), \G(t))+\delta_b\mathcal{E}_2(\F(t), \G(t)) \leq 0 \qquad \mbox{for all}\quad t\in (0,T)
\]
and
\[
\int_0^T\|\partial_t \F(t)\|_{\dot A}+\|\partial_t \G(t)\|_{\dot A}ds\leq c,
\]
for an explicit constant $c=c(\F_0,\G_0,\mu_{\pm}, \rho_{\pm}, \gamma_f, \gamma_h)$.
\end{lem}

\begin{proof}
Let $T>0$ and $(\F,\G)\in \left(C^1\left([0,T);\dot A(\TT)\right)\right)^2$ be a local solution of \eqref{eq:system3} with initial datum $(\F_0, \G_0)\in \left( \dot A(\TT)\right)^2$. We are going to show that under the size restriction \eqref{eq:size} on the initial datum $(\F_0, \G_0)$ the functional
\[
	\mathcal{E}_0(\F,\G)=  \|\F\|_{ \dot A} + \|\G\|_{\dot A}
\]
decreases in time.
To do so, we compute
\begin{align*}
\pat|\hat{\F}(k)| &=\frac{\text{Re}\left(\hat{\F}^*(k)\pat\hat{\F}(k)\right)}{|\hat{\F}(k)|},
\end{align*}
where $\F^*$ denotes the complex conjugate of $\F$.
Then,
\begin{align}\label{eq:estimate_F}
\frac{d}{dt}\|\F\|_{\dot A}\leq - \langle f_0 \rangle \left[ A_\gamma\|\F\|_{\dot{A}^4}- A \|\G\|_{\dot{A}^4} + b_\rho \|\F\|_{\dot{A}^2}-b\|\G\|_{\dot{A}^2}\right] + \|N_{1,A}\|_{ \dot A}+ \|N_{1,b}\|_{ \dot A},
\end{align}
and similarly,
\begin{align}\label{eq:estimate_G}
\frac{d}{dt}\|\G\|_{\dot A}\leq - \langle g_0 \rangle  \left[A_\mu \|\G\|_{\dot{A}^4}-A_\mu\|\F\|_{\dot{A}^4} + b_\mu \|\G\|_{\dot{A}^2}- b_\mu \|\F\|_{\dot{A}^2}\right] + \|N_{2,A}\|_{\dot A}+ \|N_{2,b}\|_{ \dot A}.
\end{align}

We are going to use the Banach algebra property for functions having zero mean and the interpolation inequality in Lemma \ref{lem:Wiener} to estimate the nonlinear terms $\|N_{i,A}\|_{\dot A}$ and $\|N_{i,b}\|_{\dot A}$, $i=1,2$.
Notice first that due to the interpolation inequalities, we can estimate 
\begin{align}\label{eq:estimate4}
\|\F\|_{\dot{A}^1}\|\F\|_{\dot{A}^3} &\leq \|\F\|_{\dot A}\|\F\|_{\dot{A}^4},
\end{align}
and 
\begin{align}\label{eq:estimate6}
\begin{split}
\|\F\|_{\dot{A}^1}\|g\|_{\dot{A}^3} &\leq \max\left\{\|\F\|_{\dot A},\|\G\|_{\dot A}\right\}\max\left\{\|\F\|_{\dot{A}^4},\|\G\|_{\dot{A}^4}\right\}\leq \mathcal{E}_0(\F,\G)\left(\|\F\|_{\dot{A}^4}+\|\G\|_{\dot{A}^4}\right).
\end{split}
\end{align}
Using, \eqref{eq:estimate4} and \eqref{eq:estimate6}, we obtain that
\begin{align}\label{eq:N1A}
\begin{split}
\|N_{1,A}\|_{\dot A}&\leq \|\F\|_{\dot{A}^1} \left(A_\gamma \|\F\|_{\dot{A}^3} + A\|\G\|_{\dot{A}^3} \right) + \|\F\|_{A}\left(A_\gamma \|\F\|_{\dot{A}^4} + A \|\G\|_{\dot{A}^4}\right)\\
&\leq 2A_\gamma\|\F\|_{\dot A}\|\F\|_{\dot{A}^4} + A \mathcal{E}_0(\F,\G)\left(\|\F\|_{\dot{A}^4}+\|\G\|_{\dot{A}^4}\right)+A\|\F\|_{\dot A}\|\G\|_{\dot{A}^4}\\
&\leq 2A_\gamma\|\F\|_{\dot A}\|\F\|_{\dot{A}^4} + 2A \mathcal{E}_0(\F,\G)\left(\|\F\|_{\dot{A}^4}+\|\G\|_{\dot{A}^4}\right)\\
&\leq (2A_\gamma+2A ) \mathcal{E}_0(\F,\G)\left(\|\F\|_{\dot{A}^4}+\|\G\|_{\dot{A}^4}\right).
\end{split}
\end{align}
Similarly,
\begin{align}\label{eq:N2A}
\begin{split}
\|N_{2,A}\|_{\dot A}&\leq A_\mu \|\G\|_{\dot{A}^1} \left( \|\F\|_{\dot{A}^3} + \|\G\|_{\dot{A}^3} \right) + A_\mu \|\G\|_{\dot A}\left( \|\F\|_{\dot{A}^4} +  \|\G\|_{\dot{A}^4}\right)\\
&\leq 4A_\mu\mathcal{E}_0(\F(t),\G(t))\left(\|\F\|_{\dot{A}^4}+\|\G\|_{\dot{A}^4}\right).
\end{split}
\end{align}
In view of
\begin{align*} %\label{eq:estimate5}
\|\F\|_{\dot{A}^1}^2 &\leq \|\F\|_{\dot A}\|\F\|_{\dot{A}^2},
\end{align*}
and
\begin{align*}
\|\F\|_{\dot{A}^1}\|g\|_{\dot{A}^1} &\leq \max\left\{\|\F\|_{A},\|\G\|_{A}\right\}\max\left\{\|\F\|_{\dot{A}^2},\|\G\|_{\dot{A}^2}\right\}\leq \mathcal{E}_0(\F,\G)\left(\|\F\|_{\dot{A}^2}+\|\G\|_{\dot{A}^2}\right),
\end{align*}
the second order nonlinearities can be estimated as
\begin{align}\label{eq:N1b}
\begin{split}
\|N_{1,b}\|_{\dot A}&\leq \|\F\|_{\dot{A}^1}\left(b_\rho \|\F\|_{\dot{A}^1} + b\|\G\|_{\dot{A}^1} \right) + \|\F\|_{\dot A}\left(b_\rho \|\F\|_{\dot{A}^2}+ b\|\G\|_{\dot{A}^2} \right)\\
&\leq (2b_\rho+2b)\mathcal{E}_0(\F,\G)(\|\F\|_{\dot{A}^2}+\|\G\|_{\dot{A}^2})
\end{split}
\end{align}
and
\begin{align}\label{eq:N2b}
\begin{split}
\|N_{2,b}\|_{\dot A}&\leq b_\mu\|\G\|_{\dot{A}^1}\left( \|\F\|_{\dot{A}^1} + \|\G\|_{\dot{A}^1} \right) + b_\mu\|\G\|_{\dot A}\left( \|\F\|_{\dot{A}^2}+ \|\G\|_{\dot{A}^2} \right)\\
&\leq 4b_\mu\mathcal{E}_0(\F,\G)(\|\F\|_{\dot{A}^2}+\|\G\|_{\dot{A}^2}).
\end{split}
\end{align}

Adding \eqref{eq:estimate_F} and \eqref{eq:estimate_G} and using \eqref{eq:N1A},  \eqref{eq:N2A}, \eqref{eq:N1b}, and \eqref{eq:N2b}, gives rise to
\begin{align}\label{ine}
\begin{split}
\frac{d}{dt}\mathcal{E}_0(\F,\G)(t)&\leq -\left[ \langle f_0 \rangle A_\gamma- \langle g_0 \rangle A_\mu - (A_\mu+2A_\gamma+2A)\mathcal{E}_0(\F(t),\G(t))\right]\|\F\|_{\dot A^4}  \\
&\quad-\left[\langle g_0 \rangle A_\mu - \langle f_0 \rangle A - (A_\mu+2A_\gamma+2A)\mathcal{E}_0(\F(t),\G(t)) \right]\|\G\|_{\dot A^4} \\
&\quad - \left[ \langle f_0 \rangle b_\rho - \langle g_0 \rangle b_\mu - \mathcal{E}_0(\F(t),\G(t))(2b_\rho+2b+4b_\mu)\right]\|\F\|_{\dot A^2}  \\
&\quad-\left[ \langle g_0 \rangle b_\mu - \langle f_0 \rangle b - \mathcal{E}_0(\F(t),\G(t))(2b_\rho+2b+4b_\mu)\right] \|\G\|_{\dot A^2}.
\end{split}
\end{align}
As a consequence of \eqref{eq:size} we obtain that
\begin{align*}
\frac{d}{dt}\mathcal{E}_0(\F(t),\G(t))\bigg{|}_{t=0}&\leq 0.
\end{align*}
By continuity there exists a time $t_0\in (0,T)$ such that
\begin{equation}\label{localdecay}
\mathcal{E}_0(\F(t),\G(t))\leq \mathcal{E}_0(\F(0),\G(0))\qquad \mbox{for all}\quad t\in [0,t_0].
\end{equation}
We want to propagate the local in time decay \eqref{localdecay} for all times $t\in [0,T)$. Let us emphasize that $\mathcal{E}_0(\F(t),\G(t))\leq \mathcal{E}_0(\F_0,\G_0)$ implies that for any $t\in [0,t_0]$ we have that
$$
\langle f_0 \rangle A_\gamma- \langle g_0 \rangle A_\mu - (A_\mu+2A_\gamma+2A)\mathcal{E}_0(\bar{f}(t),\bar{g}(t))\geq \sigma_{1,A}>0,
$$
$$
\langle g_0 \rangle A_\mu - \langle f_0 \rangle A - (A_\mu+2A_\gamma+2A)\mathcal{E}_0(\bar{f}(t),\bar{g}(t))\geq \sigma_{2,A}>0.
$$
and
\begin{align*}
\langle f_0 \rangle b_\rho - \langle g_0 \rangle b_\mu - \mathcal{E}_0(\bar{f}(t),\bar{g}(t))(2b_\rho+2b+4b_\mu)\geq \sigma_{1,b}>0,\\
\langle g_0 \rangle b_\mu - \langle f_0 \rangle b - \mathcal{E}_0(\bar{f}(t),\bar{g}(t))(2b_\rho+2b+4b_\mu)\geq \sigma_{2,b}>0.
\end{align*}
Assume that
$$
\mathcal{E}_0(\F(t),\G(t))\leq \mathcal{E}_0(\F_0,\G_0)
$$
holds for a maximal time interval $[0,t_m]$ and $t_m<T$. Then $\mathcal{E}_0(\F(t_m),\G(t_m))=\mathcal{E}_0(\F_0,\G_0)$, but again in view of \eqref{ine} this implies that 
\begin{align*}
\frac{d}{dt}\mathcal{E}_0(\F(t),\G(t))\bigg{|}_{t=t_m}&\leq 0,
\end{align*}
which is a contradiction to $t_m<T$ and we have shown that in fact
\begin{equation*} %\label{eq:Edecay}
\mathcal{E}_0(\F(t),\G(t))\leq \mathcal{E}_0(\F_0,\G_0)\quad \mbox{for all}\quad t\in [0,T).
\end{equation*}
Now, setting $\delta_A:=\min\{\sigma_{1,A}, \sigma_{2,A}\}$ and $\delta_b:= \min\{\sigma_{1,b},\sigma_{2,b}\}$, inequality \eqref{ine} reads
\begin{equation*} %\label{parabolic}
	\frac{d}{dt}\mathcal{E}_0(\F(t),\G(t))\leq -\delta_A \mathcal{E}_4(\F(t),\G(t))-\delta_b \mathcal{E}_2(\F(t),\G(t))\quad \mbox{for all}\quad t\in [0,T).
\end{equation*}
Finally, we observe that our estimates in the Wiener space \eqref{ine} and 
\begin{equation*} %\label{balance}
\mathcal{E}_0(\F(T),\G(T))+\int_0^T\delta_A \mathcal{E}_4(\F(t),\G(t))+\delta_b \mathcal{E}_2(\F(t),\G(t))dt\leq \mathcal{E}_0(\F_0,\G_0).
\end{equation*}
 guarantee that
$$
\int_0^T\|\partial_t \F(t)\|_{\dot A}+\|\partial_t \G(t)\|_{\dot A}ds\leq c,
$$
for an explicit constant $c=c(\F_0,\G_0,\mu_{\pm}, \rho_{\pm}, \gamma_f, \gamma_h)>0$. 

\end{proof}

\subsection{Existence of global weak solutions} We start be constructing a global Galerkin approximation to \eqref{eq:system3}. The a priori estimates provided by the energy functional $\mathcal{E}_0$ then allow us to pass to the limit and to obtain a global weak solution in the sense of Definition \ref{defi1}.  
We are looking for continuously differentiable functions $\F_k, \G_k$ such that problem \eqref{eq:system3} is satisfied in the weak sense when testing against functions form a $k$-dimensional subspace. For this purpose, set 
\[
	\F_k(t,x):= \sum_{|n|\leq k}\hat \F(t,n)e^{inx} \quad \mbox{and} \quad \G_k(t,x):= \sum_{|n|\leq k}\hat \G(t,n)e^{inx}\qquad \mbox{for}\quad k\in \NN.
\] 
Since $\F_0,\G_0 \in \dot A(\TT)$ their Fourier series converge and we define
\[
	\F_k(0,x):=\sum_{|n|\leq k}\hat \F_0(n)e^{inx}\quad \mbox{and} \quad \G_k(0,x):= \sum_{|n|\leq k}\hat \G_0(n)e^{inx}\qquad \mbox{for}\quad k\in \NN.
\]
Let us denote by $\{\phi_j(x)\}_{j\in \NN}:=\{e^{ijx}\}_{j\in \NN}$ an orthogonal basis of $L_2(\TT)$. For $k\in\NN$, we consider the Galerkin approximated problems
\begin{align}\label{eq:system3reg}
\begin{split}
\partial_t \F_k &= -\langle f_0 \rangle \left[ A_\gamma \partial_x^4 \F_k + A \partial_x^4 \G_k - b_\rho \partial_x^2 \F_k - b\partial_x^2 \G_k \right]+N^k_{1,A}+N^k_{1,b}, \\
\partial_t \G_k&=-\langle g_0 \rangle \left[ A_\mu \partial_x^4 \F_k + A_\mu \partial_x^4 \G_k - b_\mu \partial_x^2\F_k - b_\mu \partial_x^2 \G_k \right]+N^k_{2,A}+N^k_{2,b},
 \end{split}
\end{align}
where the nonlinear terms $N^k_{i,A}$ and $N^k_{i,b}$, $i=1,2$ are given by
\begin{align*}
N^k_{1,A}&:=-\partial_x {P_k}\left[ \F_k \left( A_\gamma \partial_x^3 \F_k + A \partial_x^3 \G_k \right)\right], %\label{N1reg} 
\\ N^k_{1,b}&:= \partial_x {P_k}\left[ \F_k \left(b_\rho \partial_x \F_k + b \partial_x \G_k\right) \right], %\label{N2reg}
\\
N^k_{2,A}&:=-\partial_x {P_k}\left[ \G_k \left( A_\mu \partial_x^3 \F_k +  A_\mu\partial_x^3 \G_k \right) \right], %\label{N3reg}
\\ N^k_{2,b}&:=  \partial_x {P_k}\left[ \F_k \left( b_\mu\partial_x \F_k +  b_\mu\partial_x \G_k\right) \right], %\label{N4reg}
\end{align*}
and $P_k$ denotes the projection on the subspace spanned by $\{\phi_{-k}, \ldots, \phi_k\}$. Testing the equations in \eqref{eq:system3reg} subsequently against $\phi_n$, $n=-k, \ldots, k$, yields a system of ordinary differential equations for the unknowns $\{\hat \F(\cdot, n),\hat \G(\cdot, n) \}$. Then, the Picard--Lindel\"of theorem implies the existence of local solutions $(\F_k,\G_k)\in C^1([0,T_k),C^\infty(\TT))$ to the approximate problem, where $T_k>0$ is the maximal time of existence.
Moreover, the statement of Lemma \ref{lem:EE} remains true for $(\F_k,\G_k)_{k\in \NN}$ and consequently the sequence of solutions exists globally and satisfies for any $T>0$ the following bounds:
\begin{align}\label{eq:BX}
	&(\F_k, \G_k)_k \mbox{ is uniformely bounded in } \left( L^\infty\left(0,T;\dot A(\TT)\right)\cap L^1\left(0,T;\dot A^4(\TT)\right) \right)^2,
\end{align}
\begin{align}
	\label{eq:BT}
	& (\partial_t\F_k, \partial_t \G_k)_k \mbox{ is uniformely bounded in }  \left(L^1\left(0,T;\dot A(\TT)\right)\right)^2.
\end{align}

From here we obtain the following convergences:
\begin{lem}\label{lem:convergences} The sequence $\left(\F_k, \G_k\right)_k$ satisfies
\begin{itemize}
\item[i)] $(\F_k, \G_k) \overset{*}{\warrow} (\F, \G)$ in $\left(L^\infty\left(0,T;L^\infty(\TT)\right)\right)^2$,
\item[] $(\pax^4\F_k, \pax^4\G_k) \overset{*}{\warrow} (\pax^4\F, \pax^4 \G)$ in $ \left(\mathcal{M}\left(0,T;L^\infty(\TT)\right)\right)^2$,
\item[ii)] $(\F_k, \G_k) \rightarrow (\F, \G)$ in $\left(L^1\left(0,T;C^{3+\alpha}(\TT)\right)\right)^2$,
\item[iii)] $(\partial_x^3\F_k, \partial_x^3\G_k) \overset{*}{\warrow} (\partial_x^3\F, \partial_x^3\G)$ in $\left(L^\frac{4}{3}\left(0,T;L^\infty(\TT)\right)\right)^2$
\item[iv)]  $(\F_k, \G_k) \warrow (\F, \G)$ in $\left(L^2\left(0,T;H^{2}(\TT)\right)\right)^2$, and
\item[] $(\F_k, \G_k)\rightarrow (\F, \G)$ in $\left(L^2\left(0,T; C^{1+\alpha}(\TT)\right)\right)^2$, 
\end{itemize}
where $\alpha \in [0,\frac{1}{2})$.
Moreover, the limit function $(\F, \G)$ possesses the regularity
\begin{align*}
	(\F,\G)\in &\Big(L^\infty\left(0,T;L^\infty(\TT)\right)\cap \mathcal{M}\left(0,T;W^{4,\infty}(\TT)\right) \\
	&\qquad \cap L^1\left(0,T;C^{3+\alpha}(\TT)\right) \cap L^\frac{4}{3}\left(0,T;W^{3,1}(\TT)\right) \cap L^2\left(0,T;\dot H^2(\TT)\right)\Big)
\end{align*}
for $\alpha \in [0,\frac{1}{2})$.
\end{lem}

\begin{proof}
\begin{itemize}
\item[i)] From \eqref{eq:BX} we obtain in particular the uniform bound of $(\F_k, \G_k)_k$ in 
\[\left(L^\infty\left(0,T;L^\infty(\TT)\right)\cap \mathcal{M}\left(0,T;W^{4,\infty}(\TT)\right)\right)^2.
\] By the Banach--Alaoglu theorem, we deduce that there exists a subsequence (not relabeled) such that
\[
	(\F_k, \G_k) \overset{*}{\warrow} (\F, \G)\mbox{ in } \left(L^\infty\left(0,T;L^\infty(\TT)\right)\right)^2,
\]
\[
(\pax^4\F_k, \pax^4\G_k) \overset{*}{\warrow} (\pax^4\F, \pax^4 \G) \mbox{ in } \left(\mathcal{M}(0,T;L^\infty(\TT))\right)^2.
\]
\item[ii)] Due \cite[Corollary 4]{simon1986compact}, \eqref{eq:BX}  together with \eqref{eq:BT} imply that
\[
	(\F_k, \G_k)_k \mbox{ is relatively compact in } \left(L^1\left(0,T; C^{3+\alpha}(\TT)\right)\right)^2.
\]
In particular, there exists a subsequence (not relabeled) such that
\[
	(\F_k, \G_k) \rightarrow (\F, \G) \mbox{ in } \left(L^1\left(0,T; C^{3+\alpha}(\TT)\right)\right)^2.
\]
\item[iii)] Notice first that $A^3(\TT)\subset W^{3,{\infty}}(\TT)$. Then the interpolation inequality in \eqref{interpolation} implies that
\[L^\infty\left(0,T; \dot A(\TT)\right)\cap L^1\left(0,T;\dot A^4(\TT)\right)\subset L^\frac{4}{3}\left(0,T;W^{3,{\infty}}(\TT)\right).\]
As in i), we deduce by the Banach--Alaoglu theorem that there exists a subsequence (not relabeled) such that
\[
(\partial_x^3\F_k, \partial_x^3 \G_k) \overset{*}{\warrow} (\partial_x^3\F, \partial_x^3\G)\mbox{ in } L^\frac{4}{3}\left(0,T;L^\infty(\TT)\right).
\]
\item[iv) ] Since $H^\alpha(\TT)\subset A^\alpha(\TT)$ for any $\alpha \in \RR^+$, the interpolation inequality for fractional Sobolev spaces implies that 
%\begin{equation*}\label{interpolation1}
%\|f\|_{\dot{H}^\alpha(\TT)}\leq \|f\|_{L^2(\TT)}^{1-\theta}\|f\|_{\dot{H}^{\frac{\alpha}{\theta}}(\TT)}^{\theta},
%\end{equation*}
%for all $0<\theta<1, \alpha\geq0$ and $f\in L^2(\TT)\cap \dot{H}^{\alpha/\theta}(\TT)$,
%we obtain that
%\begin{align*}
%\int_0^T\|f(s)\|_{\dot{H}^2(\TT)}^2ds&\leq \sup_{s\in[0,T]}\|f(s)\|_{L^2(\TT)}\int_0^T\|f(s)\|_{\dot{H}^{4}(\TT)}ds\\
%&\leq  \sup_{s\in[0,T]}\|f(s)\|_{\dot A(\TT)}\int_0^T\|f(s)\|_{\dot{\dot A}^{4}(\TT)}ds,
%\end{align*}
%which shows that 
\[
L^\infty(0,T; \dot A(\TT))\cap L^1\left(0,T;\dot A^4(\TT)\right) \subset L^2\left(0,T;\dot H^2(\TT)\right).
\]
In particular we deduce that
$$
(\F_k,\G_k)_k \mbox{ is uniformely bounded in }   \left(L^2\left(0,T;\dot{H}^2(\TT)\right)\right)^2.
$$
By the Eberlein--\v{S}mulian theorem, there exists a weakly convergent subsequence (not relabeled) such that
\[
	(\F_k,\G_k) \warrow (\F,\G) \mbox{ in } \left(L^2\left(0,T;\dot H^2(\TT)\right)\right)^2.
\]
Again, by \cite[Corollary 4]{simon1986compact} together with \eqref{eq:BT} we obtain that
\[
	(\F_k,\G_k)_ k \mbox{ is relatively compact in } \left(L^2\left(0,T;C^{1+\alpha}(\TT)\right)\right)^2
\]
for any $\alpha \in [0,\frac{1}{2})$. Hence, there exists a subsequence (not relabeled) such that 
\[
	(\F_k,\G_k) \rightarrow (\F,\G) \mbox{ in } \left(L^2\left(0,T;C^{1+\alpha}(\TT)\right)\right)^2.
\]

\end{itemize}
Notice that due to the weak and weak-* convergences stated before, we can identify the limit function $(\F, \G)$ to belong to the space
\begin{equation*}
\left( L^\infty\left(0,T;L^\infty(\TT)\right)\cap \mathcal{M}\left(0,T;W^{4,\infty}(\TT)\right)\cap L^\frac{4}{3}\left(0,T;W^{3,{\infty}}(\TT)\right)\right)^2.
\end{equation*} 
 Together with the strong convergence in ii), the stated regularity for the limit function is obtained.
\end{proof}

%
%We recall now the following result:
%
%\begin{lem}[\cite{simon1986compact}, Theorem 4]
%Assume 
%$$
%X\subset\subset B,
%$$
%$$
%\{h_n\}\subset L^\infty(0,T;B)\cap L^1_{loc}(0,T,X), 
%$$
%$$
%\int_{t_1}^{t_2}\|h_n(s+\delta)-h_n(s)\|_{B}ds \text{ tends uniformly to $0$ as }\delta\rightarrow 0\;\; \forall\,0<t_1<t_2<T.
%$$
%Then $\{h_n\}$ is relatively compact in $L^2(0,T;B)$.
%\end{lem}
%Equipped with this Lemma, we can prove the compactness of the approximate solutions $(\F_k,\G_k)$. We take
%$$
%X=H^2,\; B=L^2.
%$$
%The only thing we have to check is the third hypothesis. Fixing $0<t_1<t_2<T$, we have that
%\begin{align*}
%\int_{t_1}^{t_2}\|\F_k(s+\delta)-\F_k(s)\|_{L^2}ds&\leq \sqrt{T}\sqrt{\int_{t_1}^{t_2}\|\F_k(s+\delta)-\F_k(s)\|_{L^2}^2ds}\\
%&\leq  \sqrt{T}C(\F_0,\G_0)\sqrt{\int_{t_1}^{t_2}\|\F_k(s+\delta)-\F_k(s)\|_{L^1} ds}\\
%&\leq \sqrt{T}C(\F_0,\G_0)\sqrt{\delta\int_\TT\int_{t_1}^{t_2}\int_{0}^{1}|\partial_t\F_k(s+\lambda\delta)| d\lambda ds dx}.
%\end{align*}
%We take $\delta(T,t_2)$ such that
%$$
%t_2+\delta<T.
%$$
%Then,
%\begin{align*}
%\int_{t_1}^{t_2}\|\F_k(s+\delta)-\F_k(s)\|_{L^2}ds&\leq \sqrt{T}C(\F_0,\G_0)\sqrt{\delta \int_{0}^{T}\|\partial_t\F_k(z)\|_{\dot A}  dz }\\
%&\leq  C(T,\F_0,\G_0,A_\mu, A_\gamma, b_\mu, b_\rho, b, A)\sqrt{\delta(T,t_2)}.
%\end{align*}
%In particular, 
%$$
%\int_{t_1}^{t_2}\|\F_k(s+\delta)-\F_k(s)\|_{L^2}ds\rightarrow0
%$$
%uniformly in $k$ as $\delta\rightarrow 0$. 

Equipped with the convergences in Lemma \ref{lem:convergences}, we can pass to the limit in the weak formulation of \eqref{eq:system3reg}.

\subsection{Exponential decay towards the equilibrium} 
Since the Galerkin approximation satisfies the energy inequality in Lemma \ref{lem:EE}, we can use the Poincar\'e-like inequality for zero-mean functions in Wiener spaces and deduce that
\[
	\frac{d}{dt}\mathcal{E}_0(\F_k,\G_k)(t)\leq -\delta_A \mathcal{E}_4(\F_k,\G_k)(t)-\delta_b \mathcal{E}_2(\F_k,\G_k)(t) \leq - (\delta_A + \delta_b)\mathcal{E}_0(\F,\G)(t),
\]
which implies the exponential decay
\[
	\mathcal{E}_0(\F_k,\G_k)(t)\leq 	\mathcal{E}_0(\F_0,\G_0)e^{-(\delta_A + \delta_b)t}.
\]

Due to the fact that the weak-* convergence in Banach spaces is lower semi-continuous, Lemma \ref{lem:EE} i) implies that for almost all $T>0$:
\begin{equation*}
	\mathscr{E}_0(\F,\G)(t)\leq \liminf_{k\to \infty}\mathcal{E}_0(\F_k,\G_k)(t)\leq 	\mathcal{E}_0(\F_0,\G_0)e^{-(\delta_A + \delta_b)t},
\end{equation*}
which proves the decay assertion in Theorem \ref{theorem1} b).

\begin{remark}
We see that the presents of surface tension effects improves the estimate for the exponential decay.
\end{remark}

\subsection{Uniqueness}
The uniqueness of the solution to \eqref{eq:system3} in the class of 
$$
L^1\left(0,T;\dot{A}^4(\TT)\right),
$$
can be obtained by a standard contradiction argument. For the sake of brevity we only give a sketch of the proof. First we assume that there exist two different couples of solutions $(\F_1,\G_1)$ and $(\F_2,\G_2)$. Notice that then $(\F_i,\G_i)\in \left( W^{1,1}\left(0,T; \dot A(\TT)\right) \right)^2$, $i=1,2$, and we can estimate $\mathcal{E}_0(\F_1-\F_2,\G_1-\G_2)$ as in the previous section using the smallness assumption on the initial datum $(\F_0,\G_0)$. Eventually, we arrive at the inequality
$$
\frac{d}{dt}\mathcal{E}_0(\F_1-\F_2,\G_1-\G_2)\leq c\mathcal{E}_0(\F_1-\F_2,\G_1-\G_2)\left(\mathcal{E}_4(\F_1,\G_1)+\mathcal{E}_4(\F_2,\G_2)\right),
$$ 
where $c>0$ is a constant.
The statement is then a consequence of the Gronwall inequality and the fact that
$$
(\F_1(0),\G_1(0))=(\F_2(0),\G_2(0))=(\F_0,\G_0).
$$

\bigskip

%\section{Proof of Theorem \ref{theorem2}: Global solution in Wiener spaces for the gravity driven thin film Muskat problem}
\section{Existence and decay for the gravity driven thin film Muskat system in the Wiener algebra}
When surface tension effects are neglected, system \eqref{eq:system3} reads
\begin{align*}
\partial_t \F &= \langle f_0 \rangle \left[b_\rho \partial_x^2 \F + b\partial_x^2 \G \right]+N_{1,b}, \\
\partial_t \G&=\langle g_0 \rangle \left[  b_\mu \partial_x^2 \F + b_\mu\partial_x^2 \G \right]+N_{2,b},
\end{align*}
where the nonlinear terms $N_{i,b}$, $i=1,2$, are given in \eqref{N1} and \eqref{N3}. Repeating the arguments used for the the capillary driven flow in the previous section, we obtain that \eqref{ine} is replaced by the inequality
\begin{align*} %\label{ine2}
\begin{split}
\frac{d}{dt}\mathcal{E}_0(\F,\G)(t)&\leq - \left[ \langle f_0 \rangle b_\rho - \langle g_0 \rangle b_\mu - \mathcal{E}_0(\F,\G)(2b_\rho+2b+4b_\mu)\right]\|\F\|_{\dot {A}^2} \\
&\quad-\left[ \langle g_0 \rangle b_\mu - \langle f_0 \rangle b - \mathcal{E}_0(\F,\G)(2b_\rho+2b+4b_\mu)\right] \|\G\|_{\dot{A}^2}.
\end{split}
\end{align*}
From this point on, the proof follows by the same techniques as in the previous section.

\bigskip

%\section{Proof of Theorem \ref{theorem3}: Global solution in Sobolev spaces for the capillary driven thin film Muskat problem}
\section{Existence and decay for the capillary driven thin film Muskat system in Sobolev spaces}\label{section6}

Clearly, the assumptions of Theorem \ref{theorem3} guarantee the existence of a global weak solution as in Theorem \ref{theorem1}. In this section, we show that under the condition that the initial datum belongs to the Sobolev space $\dot H^2(\TT)$ and satisfies additional size restrictions (in the Wiener algebra $A(\TT)$), we can propagate Sobolev regularity of the solution and energy estimates in Sobolev spaces.
Here, we only prove the a priori energy estimates in the Sobolev space 
\begin{equation*} %\label{eq:Sobreg}
L^\infty\left(0,T;\dot{H}^2(\TT)\right)\cap L^2\left(0,T;\dot{H}^4(\TT)\right),
\end{equation*}
  the rest of the proof being straightforward. 
 
 \medskip
 
 Let $(\F_k,\G_k)$ be a sequence of Galerkin approximations as in Theorem \ref{theorem1} corresponding to an initial datum $(\F_0,\G_0)$ (satisfying the conditions of Theorem \ref{theorem3}). We know that the sequence is uniformly (with respect to $k$) bounded in $\left( L^\infty\left(0,T;\dot A(\TT)\right)\cap L^1\left(0,T;\dot A^4(\TT)\right) \right)^2$ for all $0<T<\infty$ (cf. \eqref{eq:BX}).
 
 \medskip
 
Multiplying the first equation in \eqref{eq:system3reg} by $\pax^4 \F_{{k}}$, we find that
\begin{align}\label{eq:estimate_FH}
\begin{split}
&\frac{1}{2}\frac{d}{dt}\|\F_{{k}}\|_{\dot{H}^2}^2=\int_\TT \partial_x^4 \F_{{k}}\partial_t \F_{{k}} dx\\
&\leq- \langle f_0 \rangle A_\gamma\|\F_{{k}}\|_{\dot{H}^4}^2+\|\F_{{k}}\|_{\dot{H}^4} \langle f_0 \rangle A \|\G_{{k}}\|_{\dot{H}^4} - \langle f_0 \rangle b_\rho \|\F_{{k}}\|_{\dot{H}^3}^2+ \langle f_0 \rangle b\|\G_{{k}}\|_{\dot{H}^4}\|\F_{{k}}\|_{\dot{H}^2}\\
&\quad+\int_\TT \pax^4\F_{{k}} {P_k}\left[-\partial_x \left[ \F_{{k}} \left( A_\gamma \partial_x^3 \F_{{k}} + A \partial_x^3 \G_{{k}} \right)\right]+ \partial_x \left[ \F_{{k}} \left(b_\rho \partial_x \F_{{k}} + b \partial_x \G_{{k}}\right) \right]\right]dx.
\end{split}
\end{align}
Then, using
$$
\|h\|_{\dot{H}^4}\leq  \sqrt{2\pi}\|h\|_{\dot{W}^{4,\infty}}\leq \sqrt{2\pi}\|h\|_{\dot{A}^4}
$$
and the properties of the solution described in Theorem \ref{theorem1}, we have that \eqref{eq:estimate_FH}  can be estimated as
\begin{align*}
\frac{1}{2}\frac{d}{dt}\|\F_{{k}}\|_{\dot{H}^2}^2&\leq- \langle f_0 \rangle A_\gamma\|\F_{{k}}\|_{\dot{H}^4}^2+\|\F_{{k}}\|_{\dot{H}^4} \langle f_0 \rangle A \|\G_{{k}}\|_{\dot{H}^4} - \langle f_0 \rangle b_\rho \|\F_{{k}}\|_{\dot{H}^3}^2+ c\mathscr{E}_4(\F_{{k}},\G_{{k}})\sqrt{E_2(\F_{{k}},\G_{{k}})}\\
&\quad+\int_\TT \pax^4\F_{{k}} {P_k}\left[-\partial_x \left[ \F_{{k}} \left( A_\gamma \partial_x^3 \F_{{k}} + A \partial_x^3 \G_{{k}} \right)\right]\right]dx,
\end{align*}
where $c=c(\F_0,\G_0,\mu_{\pm}, \rho_{\pm}, \gamma_f,\gamma_h)>0.$  Integrating by parts appropriately, and using the definitions of \eqref{notation1}, \eqref{notation1b} and \eqref{notation1c}, we find that
\begin{align}
&\int_\TT \pax^4\F_{{k}} {P_k}\left[\partial_x \left[ \F_{{k}}\left( A_\gamma \partial_x^3 \F_{{k}} + A \partial_x^3 \G_{{k}} \right)\right]\right] dx
\leq A_\gamma \|\F_{{k}}\|_{L^\infty} \|\F_{{k}}\|_{\dot H^4}^2 + A \|\F_{{k}}\|_{L^\infty} \|\F_{{k}}\|_{\dot H^4}\|\G_{{k}}\|_{\dot H^4} \nonumber \\
\label{eq:F1}
&\qquad\qquad\quad  +\frac{A_\gamma}{4}\mathscr{E}_4(\F_{{k}},\G_{{k}})E_2(\F_{{k}},\G_{{k}}) + A \int_{\TT}\partial_x \F_{{k}} \partial_x^4 \F_{{k}} \partial_x^3 \G_{{k}} \,dx\\
&\qquad\qquad\leq  \sqrt{2}A_\gamma  \|\F_{{k}}\|_{\dot H^4}\mathscr{E}_0(\F_{{k}},\G_{{k}})\sqrt{E_4(\F_{{k}}, \G_{{k}})}  +\frac{A_\gamma}{4}\mathscr{E}_4(\F_{{k}},\G_{{k}})E_2(\F_{{k}},\G_{{k}})\nonumber\\
&\qquad\qquad\quad + A \int_{\TT}\partial_x \F_{{k}} {P_k}(\partial_x^4 \F_{{k}} \partial_x^3 \G_{{k}}) \,dx.\nonumber
\end{align}

We recall the Kolmogorov-Landau inequality 
\begin{equation}\label{kolmogorov-landau}
	\|\partial_xh\|_{L^\infty}^2 \leq 2 \|h\|_{L^\infty}\|\partial_x^2h\|_{L^\infty}.
\end{equation}
Applying the latter twice yields
\[
	\|\partial_x^2h\|_{L^\infty} \leq 4 \|h\|_{L^\infty}^\frac{1}{2}\|\partial_x^4h\|_{L^\infty}^\frac{1}{2}.
\]
The above inequality is going to  be useful when estimating the last integral in \eqref{eq:F1}.
This integral is delicate because of the lack of integration by parts procedure. Instead, we use that
\[
	\|\partial_x h\|_{L^4}^2\leq 3 \|h\|_{{L^\infty}}\|h\|_{\dot H^2} \qquad \mbox{and} \qquad \|\partial_x^3 h\|_{L^4}^2 \leq 3\|\partial_x^2h\|_{{L^\infty}} \|h \|_{\dot H^4}
\]
together with the interpolation inequality for Sobolev spaces to estimate the remaining integral as
\begin{align}\label{eq:estimatecross}
\begin{split}
\left|\int_\TT \partial_x  \F{P_k}(\pax^4 \F   \partial_x^3 \G) dx\right|&\leq \|\F_{{k}} \|_{\dot{H}^{4}}\|\pax \F_{{k}}\|_{L^4}\|\pax^3\G_{{k}}\|_{L^4}\\
&\leq 3\|\F_{{k}}\|_{\dot{H}^{4}}\sqrt{\|\F_{{k}}\|_{L^\infty}\|\F_{{k}}\|_{\dot{H}^{2}}}\sqrt{\|\pax^2\G_{{k}}\|_{L^\infty}\|\G_{{k}}\|_{\dot{H}^{4}}}\\
&\leq 12\|\F_{{k}}\|_{\dot{H}^{4}}\sqrt[4]{E_4(\F_{{k}},\G_{{k}})}\sqrt{\mathscr{E}_0(\F_{{k}},\G_{{k}})\|\F_{{k}}\|_{\dot{H}^{2}}}\sqrt[4]{\mathscr{E}_0(\F_{{k}},\G_{{k}})\mathscr{E}_4(\F_{{k}},\G_{{k}})}.
\end{split}
\end{align}
Collecting all the previous estimates, we find that
\begin{align}\label{eq:estimate_FH2}
\begin{split}
\frac{1}{2}\frac{d}{dt}\|\F_{{k}}\|_{\dot{H}^2}^2&\leq- \langle f_0 \rangle A_\gamma\|\F_{{k}}\|_{\dot{H}^4}^2+\|\F_{{k}}\|_{\dot{H}^4} \langle f_0 \rangle A \|\G_{{k}}\|_{\dot{H}^4} - \langle f_0 \rangle b_\rho \|\F_{{k}}\|_{\dot{H}^3}^2+ c\mathscr{E}_4(\F_{{k}},\G_{{k}})\sqrt{E_2(\F_{{k}},\G_{{k}})}\\
&\quad+\sqrt{2}A_\gamma\mathscr{E}_0(\F_{{k}},\G_{{k}})\|\F_{{k}}\|_{\dot{H}^4}\sqrt{E_4(\F_{{k}},\G_{{k}})}+\frac{A_\gamma}{4}\mathscr{E}_4(\F_{{k}},\G_{{k}})E_2(\F_{{k}},\G_{{k}})\\
&\quad+{12}A\|\F_{{k}}\|_{\dot{H}^{4}}\sqrt[4]{E_4(\F_{{k}},\G_{{k}})}\sqrt{\mathscr{E}_0(\F_{{k}},\G_{{k}})\|\F_{{k}}\|_{\dot{H}^{2}}}\sqrt[4]{\mathscr{E}_0(\F_{{k}},\G_{{k}})\mathscr{E}_4(\F_{{k}},\G_{{k}})}.
\end{split}
\end{align}
Similarly, we multiply the second equation in \eqref{eq:system3} by $\pax^4 \G_{{k}}$ and we find that
\begin{align}\label{eq:estimate_GH}
\begin{split}
\frac{1}{2}\frac{d}{dt}\|\G_{{k}}\|_{\dot{H}^2}^2&=\langle g_0 \rangle A_\mu \|\F_{{k}}\|_{\dot{H}^4}\|\G_{{k}}\|_{\dot{H}^4} - \langle g_0 \rangle A_\mu \|\G_{{k}}\|_{\dot{H}^4}^2 + C\sqrt{E_2(\F_{{k}},\G_{{k}})}\mathscr{E}_4(\F_{{k}},\G_{{k}}) - \langle g_0 \rangle b_\mu\|\G_{{k}}\|_{\dot{H}^3}^2 \\
&\quad +{\sqrt{2}A_{\mu}\|\G_{{k}}\|_{\dot{H}^4}\mathscr{E}_0(\F_{{k}},\G_{{k}})\sqrt{E_4(\F_{{k}},\G_{{k}})}}+\frac{A_\gamma}{4}\mathscr{E}_4(\F_{{k}},\G_{{k}})E_2(\F_{{k}},\G_{{k}})\\
&\quad+{12}A_\mu \|\G_{{k}}\|_{\dot{H}^{4}}\sqrt[4]{E_4(\F_{{k}},\G_{{k}})}\sqrt{\mathscr{E}_0(\F_{{k}},\G_{{k}})\|\G_{{k}}\|_{\dot{H}^{2}}}\sqrt[4]{\mathscr{E}_0(\F_{{k}},\G_{{k}})\mathscr{E}_4(\F_{{k}},\G_{{k}})}.
\end{split}
\end{align}

Summing up \eqref{eq:estimate_FH2} and \eqref{eq:estimate_GH} and using $\sqrt{x}\leq 1+x$, we obtain that
\begin{align*}
\frac{1}{2}\frac{d}{dt}E_2(\F_{{k}},\G_{{k}})&\leq\|\F_{{k}}\|_{\dot{H}^4} \langle f_0 \rangle A \|\G_{{k}}\|_{\dot{H}^4}+\langle g_0 \rangle A_\mu \|\F_{{k}}\|_{\dot{H}^4}\|\G_{{k}}\|_{\dot{H}^4} - \langle f_0 \rangle A_\gamma\|\F_{{k}}\|_{\dot{H}^4}^2 - \langle g_0 \rangle A_\mu \|\G_{{k}}\|_{\dot{H}^4}^2\\
&\quad + c(E_2(\F_{{k}},\G_{{k}})+1)\mathscr{E}_4(\F_{{k}},\G_{{k}}) - \langle g_0 \rangle b_\mu\|\G_{{k}}\|_{\dot{H}^3}^2- \langle f_0 \rangle b_\rho \|\F_{{k}}\|_{\dot{H}^3}^2 \nonumber\\
&\quad +{\sqrt{2}A_{\mu}}\|\G_{{k}}\|_{\dot{H}^4}\mathscr{E}_0(\F_{{k}},\G_{{k}})\sqrt{E_4(\F_{{k}},\G_{{k}})}+{\sqrt{2}A_\gamma}\mathscr{E}_0(\F_{{k}},\G_{{k}})\|\F_{{k}}\|_{\dot{H}^4}\sqrt{E_4(\F_{{k}},\G_{{k}})}\nonumber\\
&\quad+{12}A_\mu \|\G_{{k}}\|_{\dot{H}^{4}}\sqrt[4]{E_4(\F_{{k}},\G_{{k}})}\sqrt{\mathscr{E}_0(\F_{{k}},\G_{{k}})\|\G_{{k}}\|_{\dot{H}^{2}}}\sqrt[4]{\mathscr{E}_0(\F_{{k}},\G_{{k}})\mathscr{E}_4(\F_{{k}},\G_{{k}})}\\
&\quad+{12}A\|\F_{{k}}\|_{\dot{H}^{4}}\sqrt[4]{E_4(\F_{{k}},\G_{{k}})}\sqrt{\mathscr{E}_0(\F_{{k}},\G_{{k}})\|\F_{{k}}\|_{\dot{H}^{2}}}\sqrt[4]{\mathscr{E}_0(\F_{{k}},\G_{{k}})\mathscr{E}_4(\F_{{k}},\G_{{k}})}.
\end{align*}
Young's inequality implies that

\begin{align}\label{eq:Y}
\begin{split}
K\|\G\|_{\dot{H}^{4}}\sqrt[4]{E_4(\F_k,\G_k)}\mathscr{E}_0(\F_k,\G_k)^{3/4}&\sqrt{\|\G_k\|_{\dot{H}^{2}}}\sqrt[4]{\mathscr{E}_4(\F_k,\G_k)}\\
&\leq\frac{3}{4}E_4(\F_k,\G_k)\mathscr{E}_0(\F_k,\G_k)+\frac{K^4}{4}\|\G_k\|_{\dot{H}^{2}}^2\mathscr{E}_4(\F_k,\G_k),\\
K\|\F_k\|_{\dot{H}^{4}}\sqrt[4]{E_4(\F_k,\G_k)}\mathscr{E}_0(\F_k,\G_k)^{3/4}&\sqrt{\|\F_k\|_{\dot{H}^{2}}}\sqrt[4]{\mathscr{E}_4(\F_k,\G_k)}\\
&\leq\frac{3}{4}E_4(\F_k,\G_k)\mathscr{E}_0(\F_k,\G_k)+\frac{K^4}{4}\|\F_k\|_{\dot{H}^{2}}^2\mathscr{E}_4(\F_k,\G_k)
\end{split}
\end{align}
for any positive constant $K$. We find that
\[
\|\F_{{k}}\|_{\dot{H}^4}  \|\G_{{k}}\|_{\dot{H}^4}(\langle f_0 \rangle A+\langle g_0 \rangle A_\mu)\leq \frac{E_4(\F_{{k}},\G_{{k}})}{2}(\langle f_0 \rangle A+\langle g_0 \rangle A_\mu),
\]
and
\[
\sqrt{2}({A_{\mu}}\|\G_{{k}}\|_{\dot{H}^4}+{A_\gamma}\|\F_{{k}}\|_{\dot{H}^4})\mathscr{E}_0(\F_{{k}},\G_{{k}})\sqrt{E_4(\F_{{k}},\G_{{k}})}\leq \sqrt{2}({A_\mu + A_\gamma})\mathscr{E}_0(\F_{{k}},\G_{{k}})E_4(\F_{{k}},\G_{{k}}).
\]
Eventually, we can further regroup terms and conclude that
\begin{align*}
\frac{1}{2}\frac{d}{dt}E_2(\F_{{k}},\G_{{k}})&\leq\frac{E_4(\F_{{k}},\G_{{k}})}{2}(\langle f_0 \rangle A+\langle g_0 \rangle A_\mu) - \langle f_0 \rangle A_\gamma\|\F_{{k}}\|_{\dot{H}^4}^2 - \langle g_0 \rangle A_\mu \|\G_{{k}}\|_{\dot{H}^4}^2\\
&\quad + {c\left(1+E_2(\F,\G)\right)}\mathscr{E}_4(\F_{{k}},\G_{{k}}) - \langle g_0 \rangle b_\mu\|\G_{{k}}\|_{\dot{H}^3}^2- \langle f_0 \rangle b_\rho \|\F_{{k}}\|_{\dot{H}^3}^2 \nonumber\\
&\quad +\left(\sqrt{2}A_\gamma +\sqrt{2} A_{\mu}+\frac{9}{4}A+\frac{9}{4}A_\mu\right)\mathscr{E}_0(\F_{{k}},\G_{{k}})E_4(\F_{{k}},\G_{{k}}).
\end{align*}
Using that
$$
\mathscr{E}_s(\F,\G)\leq \mathcal{E}_s(\F,\G),
$$
the additional assumptions in Theorem \ref{theorem3}, which are given by
$$
\langle g_0 \rangle A_\mu-\frac{\langle f_0 \rangle A+\langle g_0 \rangle A_\mu}{2}-\left(\sqrt{2}A_\gamma +\frac{9}{4}A+\left(\sqrt{2}+\frac{9}{4}\right)A_\mu\right)\mathcal{E}_0(\F_0,\G_0)>0,
$$
$$
\langle f_0 \rangle A_\gamma-\frac{\langle f_0 \rangle A+\langle g_0 \rangle A_\mu}{2}-\left(\sqrt{2}A_\gamma +\frac{9}{4}A+\left(\sqrt{2}+\frac{9}{4}\right)A_\mu\right)\mathcal{E}_0(\F_0,\G_0)>0,
$$
together with the decay of Theorem \ref{theorem1} imply the existence of $0<\delta_i\ll 1, i=1,2$ small enough such that
\begin{align*}
\frac{d}{dt}E_2(\F_{{k}},\G_{{k}})+\delta_1 E_4(\F_{{k}},\G_{{k}})+\delta_2 E_3(\F_{{k}},\G_{{k}})&\leq c(E_2(\F_{{k}},\G_{{k}})+1)\mathscr{E}_4(\F_{{k}},\G_{{k}}).
\end{align*}
{Due to Theorem \ref{theorem1}, we know that the Galerkin approximation $(\F_{{k}},\G_{{k}})$ satisfies
\[
(\F_{k},\G_{k})\in\left(L^1\left(0,T;\dot{A}^{4}(\TT)\right)\right)^2.
\]
In particular, $\mathscr{E}_4(\F_{{k}}(s),\G_{{k}}(s))$ is integrable and
\[
\int_0^T \mathscr{E}_4(\F_{{k}}(s),\G_{{k}}(s))ds\leq c\mathcal{E}_0(\F_0,\G_0) \qquad \mbox{for any}\quad T>0.
\]
}
Therefore, {using the Gronwall inequality,} we conclude that
\[
	E_2(\F_{{k}},\G_{{k}})(t)\leq E_2(\F_0,\G_0)\left(e^{ c\int_0^t \mathscr{E}_4(\F_{{k}}(s),\G_{{k}}(s))\,ds}-1\right)\leq  c E_2(\F_0,\G_0) \qquad \mbox{for any}\quad T>0
\]
 and we obtain the desired Sobolev regularity for the Galerkin approximation. The exponential decay in $H^s(\TT)$ for $0\leq s<2$ is guaranteed by standard Sobolev interpolation.
Thus,
\begin{align*} %\label{eq:BX2}
	&(\F_k, \G_k)_k \mbox{ is uniformely bounded in } \left( L^\infty\left(0,T;\dot H^2(\TT)\right)\cap L^2\left(0,T;\dot H^4(\TT)\right) \right)^2,
\end{align*}
\begin{align*} %\label{eq:BX3}
	&(\pat \F_k, \pat \G_k)_k \mbox{ is uniformely bounded in } \left(L^2\left(0,T; L^2(\TT)\right) \right)^2.
\end{align*}
 Then, a standard argument as in Lemma \ref{lem:convergences} guarantees that the there exists a subsequence of $(\F_k, \G_k)_k$ (not relabeled) satisfying
\begin{itemize}
\item[i)] $(\F_k, \G_k) \warrow (\F, \G)$ in $\left(L^\infty(0,T;\dot H^2(\TT))\right)^2$,
\item[ii)] $(\F_k, \G_k) \warrow (\F, \G)$ in $\left(L^2\left(0,T;\dot{H}^{4}(\TT)\right)\right)^2$,
\item[iii)] $(\F_k, \G_k) \rightarrow (\F, \G)$ in $\left(L^2\left(0,T;C^{3+\alpha}(\TT)\right)\cap C\left([0,T];C^{1+\alpha}(\TT)\right)\right)^2$ for $\alpha \in [0,\frac{1}{2})$,
\item[iv)] $(\partial_t \F_k, \partial_t \G_k)\warrow (\partial_t \F, \partial_t \G)$ in $\left(L^2\left(0,T; L^2(\TT)\right) \right)^2$.
\end{itemize}

Moreover, in view of the convergences above the limit function $(\F, \G)$ possesses the regularity
\begin{align*}
	(\F,\G)\in &\Big(C\left(0,T;\dot H^2(\TT)\right)\cap L^2\left(0,T;\dot H^{4}(\TT)\right)\Big)^2.
\end{align*}

Finally, we can pass to the limit as in Theorem \ref{theorem1}.

\bigskip

\section{Existence and decay for the gravity driven thin film Muskat system in Sobolev spaces}
The proof of Theorem \ref{theorem4} essentially follows the lines in the previous section. To obtain the a priori estimate in Sobolev spaces, we multiply the equation for $\F$ by $-\pax^2\F$ and integrate by parts to obtain the inequality
\begin{align*}
\frac{1}{2}\frac{d}{dt}\|\F\|_{\dot{H}^1}^2 &\leq -\langle f_0 \rangle b_\rho \|\F\|_{\dot{H}^2}^2 + \langle f_0 \rangle b\|\F\|_{\dot{H}^2}\|\G\|_{\dot{H}^2}+b_\rho \|\F\|_{\dot{H}^2}^2\|\F\|_{L^\infty}+\frac{5b}{2}\|\F\|_{L^\infty}\|\F\|_{\dot{H}^2}\|\G\|_{\dot{H}^2}.
\end{align*}
Analogously, multiplying the equation for $\G$ by $-\pax^2\G$ and integrating by parts yields that
\begin{align*}
\frac{1}{2}\frac{d}{dt}\|\G\|_{\dot{H}^1}^2 &\leq  \langle g_0\rangle b_\mu \|\F\|_{\dot{H}^2}\|\G\|_{\dot{H}^2}  - \langle g_0\rangle b_\mu \|\G\|_{\dot{H}^2}^2+b_\mu \|\G\|_{\dot{H}^2}^2\|\G\|_{L^\infty}+\frac{5b_\mu}{2}\|\G\|_{L^\infty}\|\F\|_{\dot{H}^2}\|\G\|_{\dot{H}^2}. 
\end{align*}
Taking the sum of both inequalities, we obtain that
\begin{align*}
\frac{1}{2}\frac{d}{dt}E_1(\F,\G) &\leq  \frac{1}{2}\left(\langle g_0\rangle b_\mu+\langle f_0 \rangle b\right)E_2(\F,\G) -\langle f_0 \rangle b_\rho \|\F\|_{\dot{H}^2}^2 - \langle g_0\rangle b_\mu \|\G\|_{\dot{H}^2}^2\\
&\quad+\left(b_\rho+b_\mu+\frac{5b_\mu}{2}+\frac{5b}{2}\right) E_2(\F,\G)\mathscr{E}_0(\F,\G). 
\end{align*}
From here on, we can conclude the statement using the previous ideas and Theorem \ref{theorem2}.

\bigskip

\section{Existence and decay for the thin film Stokes system in the Wiener algebra}
The proof of Theorem \ref{theorem5} is similar to the proof of the corresponding results for the two-phase thin film Muskat problem in Theorem \ref{theorem1} and Theorem \ref{theorem2}. Therefore, we only provide the energy estimates. Recalling the definition of $\zeta$ from the statement of Theorem \ref{theorem5}, we compute
\begin{equation}\label{eq:estiamte_F_S}
\begin{split}
\frac{d}{dt}\|\F\|_{\dot A}&\leq -\left(2\rho \langle f_0 \rangle^3 +3\langle f_0 \rangle^2\langle g_0 \rangle\right)\|\F\|_{\dot{A}^{\zeta+1}}+(2\langle f_0 \rangle^3+3\langle f_0 \rangle^2\langle g_0 \rangle) \|\G\|_{\dot{A}^{\zeta+1}}\\
&\qquad + \|N_1\|_{\dot A} + \|N_2\|_{\dot A}
\end{split}
\end{equation}
and
\begin{align*} %\label{eq:estimate_G_S}
\begin{split}
\frac{d}{dt}\|\G\|_{\dot A}&\leq\left(2\mu \langle g_0 \rangle^3+3\rho \langle f_0 \rangle^2\langle g_0 \rangle+6\langle f_0 \rangle \langle g_0 \rangle^2\right) \|\F\|_{\dot{A}^{\zeta+1}} \\
&\qquad - \left(2\mu \langle g_0 \rangle^3 + 3 \langle f_0 \rangle^2\langle g_0 \rangle + 6\langle f_0 \rangle \langle g_0 \rangle^2 \right) \|\G\|_{\dot{A}^{\zeta+1}}+\|N_3\|_{\dot A}+\|N_4\|_{\dot A}. 
\end{split}
\end{align*}
The nonlinear terms $N_i$, $i=1\ldots 4$, can be estimated as follows:
\begin{align*} %\label{N1stokesest}
\begin{split}
\|N_1\|_{\dot A}&\leq \mathcal{E}_0(\F,\G)\bigg{[}(3+2\rho)\mathcal{E}_0(\F,\G)^2 +  \mathcal{E}_0(\F,\G)(6(1+\rho)\langle f_0 \rangle + {3} \langle g_0 \rangle)\\
&\quad +  (3+6\rho)\langle f_0 \rangle^2 +6\langle f_0 \rangle \langle g_0 \rangle\bigg{]}\|\F\|_{\dot A^{\zeta+1}}\\
&\quad+\mathcal{E}_0(\F,\G)\bigg{[}(6\rho+9)\mathcal{E}_0(\F,\G)^2+(\langle f_0 \rangle(12\rho+12)+6\langle g_0 \rangle) \mathcal{E}_0(\F,\G)\\
&\quad+6\langle f_0 \rangle \langle g_0\rangle+(3+6\rho)\langle f_0 \rangle^2\bigg{]}\mathcal{E}_{\zeta+1}(\F,\G),
\end{split}
\end{align*}
\begin{align} \label{N2stokesest}
\begin{split}
\|N_2\|_{\dot A}&\leq \mathcal{E}_0(\F,\G)\bigg{[}5\mathcal{E}_0(\F,\G)^2 +  \mathcal{E}_0(\F,\G)(12\langle f_0 \rangle +  {3}\langle g_0 \rangle)+ 9\langle f_0 \rangle^2 +6\langle f_0 \rangle \langle g_0 \rangle\bigg{]}\|\G\|_{\dot A^{\zeta+1}}\\
&+\mathcal{E}_0(\F,\G)\bigg{[}15\mathcal{E}_0(\F,\G)^2+(24\langle f_0 \rangle+6\langle g_0 \rangle) \mathcal{E}_0(\F,\G)+6\langle f_0 \rangle \langle g_0\rangle+9\langle f_0 \rangle^2\bigg{]}\mathcal{E}_{\zeta+1}(\F,\G),
\end{split}
\end{align}
\begin{align}\label{N3stokesest}
\begin{split}
\|N_3\|_{\dot A}&\leq \mathcal{E}_0(\F,\G)\bigg{[}\left(2\mu+3\rho+6 \right)\mathcal{E}_0(\F,\G)^2 + \left((6\mu+3\rho+12)\langle g_0 \rangle+6(\rho+1)\langle f_0 \rangle\right)\mathcal{E}_0(\F,\G) \\
&\quad +6\mu \langle g_0\rangle^2 + 3\rho(\langle f_0 \rangle^2 +2 \langle f_0 \rangle \langle g_0 \rangle)  + 6( \langle g_0 \rangle^2 +2 \langle g_0 \rangle \langle f_0 \rangle)\bigg{]}\|\F\|_{\dot{A}^{\zeta+1}}\\
&\quad+\mathcal{E}_0(\F,\G)\bigg{[}\left(6\mu+9\rho+18 \right)\mathcal{E}_0(\F,\G)^2 + 6\mu \langle g_0\rangle^2 + 3\rho(\langle f_0 \rangle^2 +2 \langle f_0 \rangle \langle g_0 \rangle)  + 6 \langle g_0 \rangle^2 \\
&\quad +{12} \langle g_0 \rangle \langle f_0 \rangle+\left((12\mu+6\rho+24)\langle g_0 \rangle+12\rho\langle f_0 \rangle+12\langle f_0 \rangle\right)\mathcal{E}_0(\F,\G)\bigg{]}\mathcal{E}_{\zeta+1}(\F,\G),
\end{split}
\end{align}
\begin{align}\label{N4stokesest}
\begin{split}
\|N_4\|_{\dot A}&\leq \mathcal{E}_0(\F,\G)\bigg{[}\left(2\mu+9 \right)\mathcal{E}_0(\F,\G)^2 + \left((6\mu+15)\langle g_0 \rangle+12\langle f_0 \rangle\right)\mathcal{E}_0(\F,\G) \\
&\qquad +6\mu \langle g_0\rangle^2 + 3(\langle f_0 \rangle^2 +2 \langle f_0 \rangle \langle g_0 \rangle)  + 6( \langle g_0 \rangle^2 +2 \langle g_0 \rangle \langle f_0 \rangle)\bigg{]}\|\G\|_{\dot{A}^{\zeta+1}}\\
&\qquad+\mathcal{E}_0(\F,\G)\bigg{[}\left(6\mu+27 \right)\mathcal{E}_0(\F,\G)^2 + 6\mu \langle g_0\rangle^2 + 3(\langle f_0 \rangle^2 +2 \langle f_0 \rangle \langle g_0 \rangle)  + 6\langle g_0 \rangle^2 \\
&\qquad +{12} \langle g_0 \rangle \langle f_0 \rangle+\left((12\mu+30)\langle g_0 \rangle+24\langle f_0 \rangle\right)\mathcal{E}_0(\F,\G)\bigg{]}\mathcal{E}_{\zeta+1}(\F,\G).
\end{split}
\end{align}
Collecting the above inequalities \eqref{eq:estiamte_F_S}-\eqref{N4stokesest}, we obtain that
\begin{align*}
\frac{d}{dt}\mathcal{E}_0(\F,\G)\leq -\Sigma_1(t)\|\F\|_{\dot{A}^{\zeta+1}}-\Sigma_2(t)\|\G\|_{\dot{A}^{\zeta+1}},
\end{align*}
where
\begin{align*}
\Sigma_1(t)&=2\rho \langle f_0 \rangle^3 +3\langle f_0 \rangle^2\langle g_0 \rangle(1-\rho)-\left(2\mu \langle g_0 \rangle^3+6\langle f_0 \rangle \langle g_0 \rangle^2\right)\\
&\quad-\mathcal{E}_0(\F,\G)\bigg{[}({78}+20\rho+14\mu)\mathcal{E}_0(\F,\G)^2+(\langle f_0 \rangle(36\rho+84)+({81}+30\mu+{9}\rho)\langle g_0 \rangle)\mathcal{E}_0(\F,\G)\bigg{]}\\
&\quad-\mathcal{E}_0(\F,\G)\bigg{[}(18\mu+{18}) \langle g_0\rangle^2 + ({18\rho}+ {18})\langle f_0 \rangle^2 +({12\rho}+{60}) \langle g_0 \rangle \langle f_0 \rangle\bigg{]},
\end{align*}
\begin{align*}
\Sigma_2(t)&=2\mu \langle g_0 \rangle^3  + 6\langle f_0 \rangle \langle g_0 \rangle^2-2\langle f_0 \rangle^3 \\
&\quad-\mathcal{E}_0(\F,\G)\bigg{[}(14\mu+15\rho+83)\mathcal{E}_0(\F,\G)^2+\left((30\mu+6\rho+{84})\langle g_0 \rangle+(96+24\rho)\langle f_0 \rangle\right)\mathcal{E}_0(\F,\G)\bigg{]}\\
&\quad-\mathcal{E}_0(\F,\G)\bigg{[}(18\mu+18) \langle g_0\rangle^2 + (27+9\rho)\langle f_0 \rangle^2 +({6\rho}+{66}) \langle f_0 \rangle \langle g_0 \rangle\bigg{]}\,.
\end{align*}
Consequently, under the hypotheses of the theorem, there exists a constant $0<\varepsilon\ll1$ such that
$$
\frac{d}{dt}\mathcal{E}_0(\F,\G)+\varepsilon\mathcal{E}_{\zeta+1}(\F,\G) \leq 0.
$$
 The rest of the proof follows similarly as in Theorem \ref{theorem1}.

\bigskip

%\section{Proof of Theorem \ref{theorem6}: Global solution in Sobolev spaces for the thin film Stokes problem}
\section{Existence and decay for the thin film Stokes system in Sobolev spaces}\label{section9}
The proof of Theorem \ref{theorem6} is similar to the proof of Theorem \ref{theorem3}. To obtain the desired energy estimates, we multiply the equation for $\F$ in \eqref{eq:system_Stokesfinal2} by $-\pax\mathscr{D} \F$ and integrate by parts. We find that
\begin{align*}
\frac{1}{2}\frac{d}{dt}\|\F\|_{\dot{H}^{(\zeta+1)/2}}^2&\leq -\left(2\rho \langle f_0 \rangle^3 +3\langle f_0 \rangle^2\langle g_0 \rangle\right)\|\F\|_{\dot{H}^{\zeta+1}}^2 + (2\langle f_0 \rangle^3+3\langle f_0 \rangle^2\langle g_0 \rangle)\|\F\|_{\dot{H}^{\zeta+1}}\|\G\|_{\dot{H}^{\zeta+1}}\\
&\quad -\int_\TT N_1\pax\mathscr{D} \F dx-\int_\TT N_2\pax\mathscr{D} \F dx, %\label{StokesSoboF}
\end{align*}
where
\begin{align*}
\int_\TT N_1\pax\mathscr{D} \F dx=K_1+K_2+K_3 \qquad \mbox{and}\qquad 
\int_\TT N_2\pax\mathscr{D} \F dx=K_4+K_5+K_6,
\end{align*}
with with $K_1,\ldots, K_6$ being defined as:
\begin{align*}
K_1&:=\int_\TT 3(\F^2 \G + 2 \F \G \langle f_0 \rangle + \F^2 \langle g_0 \rangle + \G \langle f_0 \rangle^2 +2\F \langle f_0 \rangle \langle g_0 \rangle)(\pax\mathscr{D} \F)^2dx\nonumber\\
&\quad+\int_\TT 2\rho (\F^3 + 3 \F^2 \langle f_0\rangle + 3 \F \langle f_0 \rangle^2)(\pax\mathscr{D} \F)^2 dx,\\
K_2&:=-\frac{1}{2}\int_\TT\partial_x^2 \left[3(\F^2 \G + 2 \F \G \langle f_0 \rangle + \F^2 \langle g_0 \rangle + \G \langle f_0 \rangle^2 +2\F \langle f_0 \rangle \langle g_0 \rangle)\right](\mathscr{D} \F)^2dx,\\
K_3&:=-\frac{1}{2}\int_\TT\partial_x^2 \left[ 2\rho (\F^3 + 3 \F^2 \langle f_0\rangle + 3 \F \langle f_0 \rangle^2)\right](\mathscr{D} \F)^2 dx,\\
K_4&:=\int_\TT 3(\F^2 \G + 2 \F \G \langle f_0 \rangle + \F^2 \langle g_0 \rangle + \G \langle f_0 \rangle^2 +2\F \langle f_0 \rangle \langle g_0 \rangle) \partial_x\mathscr{D} \G  \partial_x\mathscr{D} \F dx\nonumber\\
&\quad+ \int_\TT 2(\F^3 + 3 \F^2 \langle f_0\rangle + 3 \F \langle f_0 \rangle^2)\pax \mathscr{D} \G  \partial_x\mathscr{D} \F dx,\\
K_5&:=\int_\TT\partial_x \left[3(\F^2 \G + 2 \F \G \langle f_0 \rangle + \F^2 \langle g_0 \rangle + \G \langle f_0 \rangle^2 +2\F \langle f_0 \rangle \langle g_0 \rangle)\right]\mathscr{D} \G\pax \mathscr{D} \F dx,\\
K_6&:=\int_\TT\partial_x \left[ 2 (\F^3 + 3 \F^2 \langle f_0\rangle + 3 \F \langle f_0 \rangle^2)\right] \mathscr{D} \G\pax \mathscr{D} \F dx.
\end{align*}
Similarly,
\begin{align*}
\frac{1}{2}\frac{d}{dt}\|\G\|_{\dot{H}^{(\zeta+1)/2}}^2&\leq -\left(2\mu \langle g_0 \rangle^3 + 3 \langle f_0 \rangle^2\langle g_0 \rangle + 6\langle f_0 \rangle \langle g_0 \rangle^2 \right)\|\G\|_{\dot{H}^{\zeta+1}}^2\\
&\quad + \left(2\mu \langle g_0 \rangle^3+3\rho \langle f_0 \rangle^2\langle g_0 \rangle+6\langle f_0 \rangle \langle g_0 \rangle^2\right)\|\F\|_{\dot{H}^{\zeta+1}}\|\G\|_{\dot{H}^{\zeta+1}}\nonumber\\
&\quad -\int_\TT N_3\pax\mathscr{D} \G dx-\int_\TT N_4\pax\mathscr{D} \G dx, %\label{StokesSoboG}
\end{align*}
where
\begin{align*}
\int_\TT N_3\pax\mathscr{D} \G dx=K_7+K_8+K_9\qquad \mbox{and}\qquad 
\int_\TT N_4\pax\mathscr{D} \G dx=K_{10}+K_{11}+K_{12},
\end{align*}
with $K_7,\ldots, K_{12}$ being defined as: 
\begin{align*}
K_7&:=\int_{\TT}\left(2\mu (\G^3 + 3\G^2 \langle g_0 \rangle +3 \G \langle g_0\rangle^2) + 3\rho(\F^2 \G + 2 \F \G \langle f_0 \rangle + \F^2 \langle g_0 \rangle + \G \langle f_0 \rangle^2 \right. \\
&\left. \qquad \qquad \qquad+2\F \langle f_0 \rangle \langle g_0 \rangle)\right) \partial_x \mathscr{D}\F\partial_x \mathscr{D}\G\,dx\\
&\qquad \quad+\int_{\TT}\left(6(\G^2 \F + 2 \G \F \langle g_0 \rangle + \G^2 \langle f_0 \rangle + \F \langle g_0 \rangle^2 +2\G \langle g_0 \rangle \langle f_0 \rangle)\right)\partial_x \mathscr{D}\F\partial_x \mathscr{D}\G\,dx,\\
K_8&:=\int_{\TT}\partial_x\left(2\mu (\G^3 + 3\G^2 \langle g_0 \rangle +3 \G \langle g_0\rangle^2) + 3\rho(\F^2 \G + 2 \F \G \langle f_0 \rangle + \F^2 \langle g_0 \rangle + \G \langle f_0 \rangle^2 \right. \\
&\left. \qquad \qquad \qquad+2\F \langle f_0 \rangle \langle g_0 \rangle)\right) \mathscr{D}\F\partial_x \mathscr{D}\G\,dx,\\
K_9&:=6 \int_{\TT}\partial_x\left(\G^2 \F + 2 \G \F \langle g_0 \rangle + \G^2 \langle f_0 \rangle + \F \langle g_0 \rangle^2 +2\G \langle g_0 \rangle \langle f_0 \rangle\right)\mathscr{D}\F\partial_x \mathscr{D}\G\,dx,\\
K_{10}&:=\int_{\TT}\left(2\mu (\G^3 + 3\G^2 \langle g_0 \rangle +3 \G \langle g_0\rangle^2) + 3(\F^2 \G + 2 \F \G \langle f_0 \rangle + \F^2 \langle g_0 \rangle + \G \langle f_0 \rangle^2 +2\F \langle f_0 \rangle \langle g_0 \rangle) \right.\\
&\qquad \quad\left.+ 6(\G^2 \F + 2 \G \F \langle g_0 \rangle + \G^2 \langle f_0 \rangle + \F \langle g_0 \rangle^2 +2\G \langle g_0 \rangle \langle f_0 \rangle)\right)(\partial_x \mathscr{D}\G)^2\,dx,\\
K_{11}&:=-\frac{1}{2}\int_{\TT}\partial_x^2\left(2\mu (\G^3 + 3\G^2 \langle g_0 \rangle +3 \G \langle g_0\rangle^2) + 3(\F^2 \G + 2 \F \G \langle f_0 \rangle + \F^2 \langle g_0 \rangle + \G \langle f_0 \rangle^2 \right.\\
&\left. \qquad \qquad \qquad+2\F \langle f_0 \rangle \langle g_0 \rangle\right) (\mathscr{D}\G)^2\,dx,\\
K_{12}&:=-3\int_{\TT}\partial_x^2 \left(\G^2 \F + 2 \G \F \langle g_0 \rangle + \G^2 \langle f_0 \rangle + \F \langle g_0 \rangle^2 +2\G \langle g_0 \rangle \langle f_0 \rangle\right)(\mathscr{D}\G)^2\,dx.
\end{align*}

\medskip

\textbf{Capillary driven flow:} Let us consider the case $\mathscr{D}=-\partial_x^3$ first. We estimate
\begin{align*}
K_1&\leq 3(\mathscr{E}_0(\F,\G)^3 +  \mathscr{E}_0(\F,\G)^2 \left[{2}\langle f_0 \rangle + \langle g_0 \rangle\right] + \mathscr{E}_0(\F,\G) \left[\langle f_0 \rangle^2 +2 \langle f_0 \rangle \langle g_0 \rangle\right])\|\F\|_{\dot{H}^{4}}^2\\
&\quad+\left[2\rho (\mathscr{E}_0(\F,\G)^3 + 3 \mathscr{E}_0(\F,\G)^2 \langle f_0\rangle + 3 \mathscr{E}_0(\F,\G) \langle f_0 \rangle^2)\right]\|\F\|_{\dot{H}^{4}}^2.
\end{align*}
A first estimate on $K_2$ yields
\begin{align*}
K_2&\leq\frac{3}{2}\|\pax^3 \F\|_{L^4}^2\bigg{[}\mathscr{E}_0(\F,\G)^2(2\|\F\|_{\dot{H}^2}+\|\G\|_{\dot{H}^2}) + 2 \langle f_0 \rangle \mathscr{E}_0(\F,\G)(\|\F\|_{\dot{H}^2}+\|\G\|_{\dot{H}^2})\\
&\quad\qquad \qquad \qquad + 2\mathscr{E}_0(\F,\G)\|\F\|_{\dot{H}^2} \langle g_0 \rangle + \|\G\|_{\dot{H}^2} \langle f_0 \rangle^2 +2\|\F\|_{\dot{H}^2} \langle f_0 \rangle \langle g_0 \rangle\bigg{]}\\
&\quad+{\frac{3}{2}}\|\pax^3 \F\|_{L^4}^2\bigg{[}\mathscr{E}_0(\F,\G)({4}\|\pax \F\|_{L^4}^2+{2}\|\pax \G\|_{L^4}^2) + \left( {2}\|\pax \G\|_{L^4}^2\langle f_0 \rangle + {2}\|\pax \F\|_{L^4}^2 \left(\langle g_0 \rangle + \langle f_0 \rangle\right)\right) \bigg{]}.
\end{align*}

Recalling that
\begin{equation*} %\label{Sob2}
\|\pax^3 h\|_{L_4}^2\leq 3\|\partial_x^4h\|_{L^\infty}\|h\|_{\dot H^2}
\end{equation*}
and
\begin{equation*} %\label{Sob2*}
\|\pax h\|_{L_4}^2\leq 3\|h\|_{L^\infty}\|h\|_{\dot H^2},
\end{equation*}
the term $K_2$ can further be estimated by
\[
	K_2 \leq cE_2(\F,\G)\mathscr{E}_4(\F,\G),
\]
where $c=c(\mathscr{E}_0(\F_0,\G_0), \langle f_0 \rangle, \langle g_0 \rangle)$ is a positive constant. 
In similar fashion $K_3$ can be estimated by
\[
	K_3 \leq cE_2(\F,\G)\mathscr{E}_4(\F,\G).
\]
%\begin{align*}
%K_3&\leq \rho\bigg{[}\mathcal{E}_0(\F,\G)(3\mathcal{E}_0(\F,\G)\|\F\|_{\dot{H}^2} + 6 \|\F\|_{\dot{H}^2} \langle f_0\rangle) + 3 \|\F\|_{\dot{H}^2} \langle f_0 \rangle^2\bigg{]}\|\pax^3 \F\|_{L^4}^2\\
%&\quad+\rho 6\|\pax \F\|_{L^4}^2\bigg{[}\mathcal{E}_0(\F,\G) + \langle f_0\rangle \bigg{]}\|\pax^3 \F\|_{L^4}^2
%\end{align*}
%
%
%
%Then, we have that
%\begin{align*}
%K_2&\leq\frac{9}{2}E_2(\F,\G)\bigg{[}\mathcal{E}_0(\F,\G)^2\frac{5}{2} + 3 \langle f_0 \rangle \mathcal{E}_0(\F,\G) + 2\mathcal{E}_0(\F,\G)\langle g_0 \rangle + \frac{1}{2} \langle f_0 \rangle^2 +2 \langle f_0 \rangle \langle g_0 \rangle\bigg{]}\mathcal{E}_4(\F,\G)\\
%&\quad+27\mathscr{E}_0(\F,\G)E_2(\F,\G)\bigg{[}\mathscr{E}_0(\F,\G)\frac{7}{4} + \frac{3}{4}\langle f_0 \rangle + \langle g_0 \rangle \bigg{]}\mathscr{E}_4(\F,\G)\\
%&\leq cE_2(\F,\G)\mathscr{E}_4(\F,\G),
%\end{align*}
Collecting these estimates, we have that
\begin{align*}
\left|\int_\TT N_1\pax\mathscr{D} \F dx \right|&\leq 3(\mathscr{E}_0(\F,\G)^3 +  \mathscr{E}_0(\F,\G)^2 \left[{2}\langle f_0 \rangle + \langle g_0 \rangle\right] + \mathscr{E}_0(\F,\G) \left[\langle f_0 \rangle^2 +2 \langle f_0 \rangle \langle g_0 \rangle\right])\|\F\|_{\dot{H}^{4}}^2\\
&\quad+\left[2\rho (\mathscr{E}_0(\F,\G)^3 + 3 \mathscr{E}_0(\F,\G)^2 \langle f_0\rangle + 3 \mathscr{E}_0(\F,\G) \langle f_0 \rangle^2)\right]\|\F\|_{\dot{H}^{4}}^2\\
&\quad+cE_2(\F,\G)\mathscr{E}_4(\F,\G). %\label{estimateN1}
\end{align*}
We continue by estimating $K_4,K_5,$ and $K_6$:
\begin{align*}
K_4&\leq 3(\mathcal{E}_0(\F,\G)^3 +  \mathcal{E}_0(\F,\G)^2 \left[{2}\langle f_0 \rangle + \langle g_0 \rangle\right] + \mathcal{E}_0(\F,\G) \left[\langle f_0 \rangle^2 +2 \langle f_0 \rangle \langle g_0 \rangle\right])\|\G\|_{\dot{H}^{4}}\|\F\|_{\dot{H}^{4}}\\
&\quad+\left[2 (\mathscr{E}_0(\F,\G)^3 + 3 \mathscr{E}_0(\F,\G)^2 \langle f_0\rangle + 3 \mathscr{E}_0(\F,\G) \langle f_0 \rangle^2)\right]\|\G\|_{\dot{H}^{4}}\|\F\|_{\dot{H}^{4}}.
\end{align*}
Let us now turn to $K_5$, which can be estimated as
\begin{align*}
K_5 &\leq 6\left(\mathscr{E}_0(\F, \G)^2+(\langle f_0 \rangle + \langle g_0 \rangle)\mathscr{E}_0(\F, \G) + \langle f_0 \rangle\langle g_0 \rangle \right)\|\partial_x \F\|_{L_4}\|\partial_x^3 \G\|_{L_4}\|\F\|_{\dot H^4} \\
&\quad + 3\left( \mathscr{E}_0(\F, \G)^2 + 2 \langle f_0 \rangle \mathscr{E}_0(\F, \G) + \langle f_0 \rangle^2\right)\|\partial_x\G\|_{L_4}\|\partial_x^3 \G\|_{L_4}\|\F\|_{\dot H^4}.
\end{align*}
The same estimate as in \eqref{eq:estimatecross} implies that
\begin{align*}
\|\partial_x \G\|_{L_4}\|\partial_x^3 \G\|_{L_4}\|\F\|_{\dot H^4}&\leq {12}\|\F\|_{\dot{H}^{4}}\sqrt[4]{E_4(\F,\G)}\sqrt{\mathscr{E}_0(\F,\G)\|\F\|_{\dot{H}^{2}}}\sqrt[4]{\mathscr{E}_0(\F,\G)\mathscr{E}_4(\F,\G)}\\
&\leq  \frac{3}{4}E_4(\F,\G){\mathscr{E}}_0(\F,\G)+{C}\|\G\|_{\dot{H}^{2}}^2 {\mathscr{E}}_4(\F,\G),
\end{align*}
by Young's inequality (as in \eqref{eq:Y}). Thus
\begin{align*}
K_5 &\leq \frac{9}{4}\left(3\mathscr{E}_0(\F, \G)^2+(4\langle f_0 \rangle + 2\langle g_0 \rangle)\mathscr{E}_0(\F, \G) + 2\langle f_0 \rangle\langle g_0 \rangle + \langle f_0 \rangle^2 \right)E_4(\F,\G){\mathscr{E}}_0(\F,\G) \\
&\quad + c E_2(\F,\G){\mathscr{E}}_4(\F,\G).
\end{align*}
Similarly, we estimate
\begin{align*}
K_6 \leq 6 \left( \mathscr{E}_0(\F, \G)^2+2\mathscr{E}_0(\F,\G)\langle f_0 \rangle +  \langle f_0 \rangle^2 \right)\|\partial_x \F\|_{L_4}\|\partial_x^3 \G\|_{L_4}\|\F\|_{\dot H^4}.
\end{align*}
Using the same estimate as above and applying Young's inequality, we obtain that
\begin{align*}
K_6 \leq \frac{9}{2} \left( \mathscr{E}_0(\F, \G)^2+2\mathscr{E}_0(\F,\G)\langle f_0 \rangle +  \langle f_0 \rangle^2 \right)E_4(\F,\G){\mathscr{E}}_0(\F,\G) +cE_2(\F,\G){\mathscr{E}}_4(\F,\G)
\end{align*}
and thus
\begin{align*}
\left| \int_{\TT} N_2 \mathscr{D}\F dx\right| & \leq 3(\mathscr{E}_0(\F,\G)^3 +  \mathscr{E}_0(\F,\G)^2 \left[{2}\langle f_0 \rangle + \langle g_0 \rangle\right] + \mathcal{E}_0(\F,\G) \left[\langle f_0 \rangle^2 +2 \langle f_0 \rangle \langle g_0 \rangle\right])\|\G\|_{\dot{H}^{4}}\|\F\|_{\dot{H}^{4}}\\
&\quad+\left[2 (\mathscr{E}_0(\F,\G)^3 + 3 \mathscr{E}_0(\F,\G)^2 \langle f_0\rangle + 3 \mathscr{E}_0(\F,\G) \langle f_0 \rangle^2)\right]\|\G\|_{\dot{H}^{4}}\|\F\|_{\dot{H}^{4}}\\
&\quad  +\frac{9}{4}\left(5\mathscr{E}_0(\F, \G)^2+(8\langle f_0 \rangle + 2\langle g_0 \rangle)\mathscr{E}_0(\F, \G) + 2\langle f_0 \rangle\langle g_0 \rangle + 3\langle f_0 \rangle^2 \right)E_4(\F,\G)\mathcal{E}_0(\F,\G) \\
&\quad +cE_2(\F,\G){\mathscr{E}}_4(\F,\G).
\end{align*}
Repeating all estimates for $K_7, \ldots K_{12}$, we obtain that
\begin{align*}
\left|\int_\TT N_3\pax\mathscr{D} \G dx\right| &\leq \left( \mathscr{E}_0(\F,\G)^3(2\mu +3\rho+6) + \mathscr{E}_0(\F,\G)^2\left( (6\mu + 3\rho+12)\langle g_0 \rangle + 6(\rho+1) \langle f_0 \rangle\right)\right. \\
&\left.\quad + \mathscr{E}_0(\F,\G)\left( 6(\mu+1) \langle g_0 \rangle^2 + 3\rho \langle f_0 \rangle^2 + (6 \rho+12) \langle f_0 \rangle \langle g_0 \rangle\right) \right)\|\F\|_{\dot H^4}\|\G\|_{\dot H^4}\\
&\quad +\frac{9}{4}\left(\mathscr{E}_0(\F,\G)^2 \left(2\mu + 3\rho +6 \right) + 4\mathscr{E}_0(\F,\G)\left((\mu + \rho +1)\langle g_0 \rangle + 4(\rho+1) \langle f_0 \rangle \right)\right. \\
& \quad \left. + 2(\mu+1)\langle g_0 \rangle^2 + \rho\langle f_0 \rangle^2 + (2\rho +4)\langle f_0 \rangle \langle g_0 \rangle\right)E_4(\F,\G)\mathcal{E}_0(\F,\G)\\
&\quad+cE_2(\F,\G){\mathscr{E}}_4(\F,\G)
\end{align*}
and
\begin{align*}
\left|\int_\TT N_4\pax\mathscr{D} \G dx\right| &\leq  \left( \mathscr{E}_0(\F,\G)^3(2\mu +9)+ \mathscr{E}_0(\F,\G)^2(6\mu \langle g_0 \rangle + 12 \langle f_0 \rangle + 15 \langle g_0 \rangle) \right.\\
& \left. \quad + \mathscr{E}_0(\F,\G)(6\mu \langle g_0 \rangle^2 + 3 \langle f_0 \rangle^2 + 18 \langle f_0 \rangle \langle g_0 \rangle+6\langle g_0 \rangle^2\right) \|\G\|_{\dot H^4}\\
&\quad  + cE_2(\F,\G){\mathscr{E}}_4(\F, \G).
\end{align*}
%\begin{align*}
%K_7 &\leq \left( \mathscr{E}_0(\F,\G)^3(2\mu +3\rho+6) + \mathscr{E}_0(\F,\G)^2\left( (6\mu + 3\rho+12)\langle g_0 \rangle + 6(\rho+1) \langle f_0 \rangle\right)\right. \\
%&\left.\qquad + \mathscr{E}_0(\F,\G)\left( 6(\mu+1) \langle g_0 \rangle^2 + 3\rho \langle f_0 \rangle^2 + (6 \rho+12) \langle f_0 \rangle \langle g_0 \rangle\right) \right)\|\F\|_{\dot H^4}\|\G\|_{\dot H^4}
%\end{align*}
%\begin{align*}
%K_8 & \leq \frac{9}{4}\left( \mathscr{E}_0(\F,\G)^2(2\mu + 3 \rho) + 4\mathscr{E}_0(\F,\G)(\mu\langle g_0 \rangle + \rho\langle f_0 \rangle +  \rho \langle g_0 \rangle) \right.\\
%&\left. \qquad  + 2 \mu \langle g_0 \rangle^2 +  \rho \langle f_0 \rangle^2 + 2 \rho \langle f_0 \rangle \langle g_0 \rangle\right)E_4(\F,\G)\mathcal{E}_0(\F,\G) + c_{8}E_2(\F,\G)\mathscr{E}_4(\F,\G)
%\end{align*}
%\begin{align*}
%K_9 &\leq \frac{9}{4}\left( 6\mathscr{E}_0(\F,\G)^2 + 4\mathscr{E}_0(\F,\G)\left( \langle g_0 \rangle +  \langle f_0 \rangle\right) + 4 \langle g_0 \rangle \langle f_0 \rangle + 2\langle g_0 \rangle^2\right)E_4(\F,\G)\mathcal{E}_0(\F,\G) \\
%&\quad + c_{9}E_2(\F,\G)\mathscr{E}_4(\F,\G)
%\end{align*}

%\begin{align*}
%K_{10}&\leq \left( \mathscr{E}_0(\F,\G)^3(2\mu +9)+ \mathscr{E}_0(\F,\G)^2(6\mu \langle g_0 \rangle + 12 \langle f_0 \rangle + 15 \langle g_0 \rangle) \right.\\
%& \left. \qquad + \mathscr{E}_0(\F,\G)(6\mu \langle g_0 \rangle^2 + 3 \langle f_0 \rangle^2 + 12 \langle f_0 \rangle \langle g_0 \rangle\right) \|\G\|_{\dot H^4}\\
%K_{11}&\leq c_{11}E_2(\F,\G)\mathscr{E}_4(\F, \G)\\
%K_{12}&\leq c_{12}E_2(\F,\G)\mathscr{E}_4(\F, \G)
%\end{align*}

Set now
\begin{align*}
\eta_1&:=(2\rho-1) \langle f_0 \rangle^3 - \frac{3}{2}(\rho-1) \langle f_0 \rangle^2 \langle g_0 \rangle - 3\langle f_0 \rangle \langle g_0 \rangle^2-\mu \langle g_0 \rangle^3 \\
&\qquad \qquad  - \mathcal{E}_0(\F_0,\G_0)^3\left (32+\frac{55}{8}\rho +\frac{15}{2}\mu \right)\\
&\qquad \qquad -\mathcal{E}_0(\F_0,\G_0)^2 \left(\langle f_0 \rangle \left(66+18 \rho \right) + \langle g_0 \rangle \left(36+\frac{21}{2}\rho+18\mu \right)\right)\\
&\qquad \qquad -\mathcal{E}_0(\F_0,\G_0)\left( \langle f_0 \rangle^2 \left(\frac{57}{4}+\frac{39}{4}\rho \right) +\langle g_0 \rangle^2 \left(\frac{27}{2}+\frac{15}{2}\mu \right) + \langle f_0 \rangle\langle g_0 \rangle \left( \frac{81}{2}+\frac{15}{2}\rho\right)   \right),\\
\eta_2&:=\mu \langle g_0 \rangle^3+3\langle f_0 \rangle \langle g_0 \rangle^2 -\frac{3}{2}(\rho-1)\langle f_0 \rangle^2 \langle g_0 \rangle - \langle f_0\rangle^3 \\
&\qquad \qquad  - \mathcal{E}_0(\F_0,\G_0)^3\left (32+\frac{55}{8}\rho +\frac{15}{2}\mu \right)\\
&\qquad \qquad -\mathcal{E}_0(\F_0,\G_0)^2 \left(\langle f_0 \rangle \left(66+18 \rho \right) + \langle g_0 \rangle \left(36+\frac{21}{2}\rho+18\mu \right)\right)\\
&\qquad \qquad -\mathcal{E}_0(\F_0,\G_0)\left( \langle f_0 \rangle^2 \left(\frac{57}{4}+\frac{39}{4}\rho \right) +\langle g_0 \rangle^2 \left(\frac{27}{2}+\frac{15}{2}\mu \right) + \langle f_0 \rangle\langle g_0 \rangle \left( \frac{81}{2}+\frac{15}{2}\rho\right)   \right).
\end{align*}
The assumptions in the theorem guarantee that $\eta_1, \eta_2>0$.
Eventually, we arrive at
\begin{align*}
\frac{1}{2}\frac{d}{dt}E_2(\F,\G)\leq -\eta E_4(\F,\G) + c E_2(\F,\G){\mathscr{E}}_4(\F,\G),
\end{align*}
where $\eta:= \min\{\eta_1, \eta_2\}$.
From here the argumentation follows the lines of  the proof of Theorem~\ref{theorem3}.

\medskip

\textbf{Gravity driven flow:} Now, we consider the case where $\mathscr{D}=\partial_x$. Here we have that
\begin{align*}
K_1&\leq \left((3+2\rho)\mathscr{E}_0(\F,\G)^3 + \mathscr{E}_0(\F,\G)^2\left((6+6\rho)\langle f_0 \rangle +3 \langle g_0 \rangle \right) \right.\\
&\left.\qquad + \mathscr{E}_0(\F,\G) \left((3+6\rho)\langle f_0 \rangle^2 + 6 \langle g_0 \rangle \langle f_0 \rangle \right)\right)\|\G\|_{\dot{H}^{2}}\|\F\|_{\dot{H}^{2}},
\end{align*}
and
\begin{align*}
K_2&\leq\frac{3}{2}\bigg{[}\mathscr{E}_0(\F,\G)^2(2\|\pax^2\F\|_{L^\infty}+\|\pax^2\G\|_{L^\infty}) + 2 \langle f_0 \rangle \mathscr{E}_0(\F,\G)(\|\pax^2\F\|_{L^\infty}+\|\pax^2\G\|_{L^\infty})\\
&\quad + 2\mathscr{E}_0(\F,\G)\|\pax^2\F\|_{L^\infty} \langle g_0 \rangle + \|\pax^2\G\|_{L^\infty} \langle f_0 \rangle^2 +2\|\pax^2\F\|_{L^\infty} \langle f_0 \rangle \langle g_0 \rangle\bigg{]}\|\pax \F\|_{L^2}^2\\
&\quad+3\|\pax \F\|_{L^\infty}\bigg{[}\mathscr{E}_0(\F,\G)(\|\pax \F\|_{L^\infty}+\|\pax \G\|_{L^\infty}) + \left( \|\pax \G\|_{L^\infty}\langle f_0 \rangle + \|\pax \F\|_{L^\infty} \langle g_0 \rangle\right) \bigg{]}\|\pax \F\|_{L^2}^2\\
&\leq c\mathscr{E}_2(\F,\G)E_1(\F,\G),
\end{align*}
where we have used the Kolmogorov-Landau inequality \eqref{kolmogorov-landau}. Analogously, we have that
$$
K_3\leq c\mathscr{E}_2(\F,\G)E_1(\F,\G).
$$
The cross term can be estimated via
\begin{align*}
K_4&\leq \left(5\mathscr{E}_0(\F,\G)^3 + \mathscr{E}_0(\F,\G)^2\left(12\langle f_0 \rangle +3 \langle g_0 \rangle \right) + \mathscr{E}_0(\F,\G) \left(9\langle f_0 \rangle^2 + 6 \langle g_0 \rangle \langle f_0 \rangle \right)\right)\|\G\|_{\dot{H}^{2}}\|\F\|_{\dot{H}^{2}}
\end{align*}
and
\[
K_5+K_6\leq c\mathscr{E}_2(\F,\G)E_1(\F,\G).
\]
Similarly, we obtain that
\begin{align*}
K_7&\leq \left(\mathscr{E}_0(\F,\G)^3 (2\mu + 3\rho +6) + \mathscr{E}_0(\F,\G)^2\left( \langle f_0 \rangle (6\rho+6)+ \langle g_0 \rangle (2\mu + 3\rho + 12)\right) \right.\\
&\qquad \left.+ \mathscr{E}_0(\F,\G)\left(3\rho \langle f_0 \rangle^2 + \langle g_0 \rangle^2(6\mu + 6)+ \langle f_0 \rangle \langle g_0 \rangle (6\rho + 12) \right)\right)\|\F\|_{\dot H^2}\|\G\|_{\dot H^2},
\end{align*}
and
\begin{align*}
K_{10}&\leq \left(\mathscr{E}_0(\F,\G)^3 (2\mu + 9) + \mathscr{E}_0(\F,\G)^2\left( 12\langle f_0 \rangle + \langle g_0 \rangle (2\mu + 15)\right) \right.\\
&\qquad \left.+ \mathscr{E}_0(\F,\G)\left(3 \langle f_0 \rangle^2 + \langle g_0 \rangle^2(6\mu + 6)+ 18\langle f_0 \rangle \langle g_0 \rangle  \right)\right)\|\F\|_{\dot H^2}\|\G\|_{\dot H^2},
\end{align*}
while
\begin{align*}
K_8+K_9+K_{11}+K_{12}\leq c\mathscr{E}_2(\F,\G)E_1(\F,\G)
\end{align*}
Setting
\begin{align*}
\kappa_1&:=(2\rho-1) \langle f_0 \rangle^3 - \frac{3}{2}(\rho-1) \langle f_0 \rangle^2 \langle g_0 \rangle - 3\langle f_0 \rangle \langle g_0 \rangle^2-\mu \langle g_0 \rangle^3 \\
&\qquad \qquad -\frac{1}{2}{\mathcal{E}}_0(\F_0,\G_0)^3\left(23+5\rho+4\mu \right)\\
&\qquad \qquad -\frac{1}{2}{\mathcal{E}}_0(\F_0,\G_0)^2\left( \langle f_0 \rangle (36+12\rho)+ \langle g_0 \rangle (33+3\rho+4\mu)\right)\\
&\qquad \qquad -\frac{1}{2}{\mathcal{E}}_0(\F_0,\G_0)\left( \langle f_0 \rangle^2(15+9\rho)+ \langle g_0 \rangle^2(12+12\mu)+ \langle f_0 \rangle \langle g_0 \rangle (42+6\rho) \right),\\
\kappa_2&:= \mu \langle g_0 \rangle^3+3\langle f_0 \rangle \langle g_0 \rangle^2 -\frac{3}{2}(\rho-1)\langle f_0 \rangle^2 \langle g_0 \rangle - \langle f_0\rangle^3 \\
&\qquad \qquad -\frac{1}{2}{\mathcal{E}}_0(\F_0,\G_0)^3\left(23+5\rho+4\mu \right)\\
&\qquad \qquad -\frac{1}{2}{\mathcal{E}}_0(\F_0,\G_0)^2\left( \langle f_0 \rangle (36+12\rho)+ \langle g_0 \rangle (33+3\rho+4\mu)\right)\\
&\qquad \qquad -\frac{1}{2}{\mathcal{E}}_0(\F_0,\G_0)\left( \langle f_0 \rangle^2(15+9\rho)+ \langle g_0 \rangle^2(12+12\mu)+ \langle f_0 \rangle \langle g_0 \rangle (42+6\rho) \right),
\end{align*}
we obtain that
\begin{align*}
\frac{1}{2}\frac{d}{dt}E_1(\F, \G)\leq -\kappa E_2(\F,\G) + c\mathscr{E}_2(\F,\G)E_1(\F,\G),
\end{align*}
where $\kappa:= \min\{\kappa_1,\kappa_2\}$ is a positive constant in view of the assumptions of the theorem. From here we can conclude as before.

\bigskip

\section{Conclusion}
We have proved several global well-posedness results for two strongly coupled systems of degenerate quasilinear parabolic partial differential equations. We established conditions on the initial data and the physical parameters ensuring the global existence of solutions to the thin film Muskat and the thin film Stokes systems. One of the advantages of our approach is that the size restrictions required on the initial data are \emph{explicit} (in particular, they do not depend on universal constants coming from functional inequalities). Remarkably, these size restrictions only effect a very weak norm (in the Wiener Algebra $A(\TT)$). Furthermore, they are $O(1)$ of the physical parameters in the problem. In this respect, we improved the results in \cite{escher2012modelling, escher2013thin} where the size restrictions were assumed in higher order Sobolev norms and not explicit. Eventually, we would like to point out that our technique does not rely on a gradient flow structure or a particular form of an energy functional. Therefore, they are adaptable also to other systems of partial differential equations.

\bigskip

\section*{Acknowledgements}RGB was partially supported by the LABEX MILYON (ANR-10-LABX-0070) of Universit\'e de Lyon, within the program ``Investissements d'Avenir'' (ANR-11-IDEX-0007) operated by the French National Research Agency (ANR). GB recognizes the support of grant no. 250070 of the Research Council of Norway. Part of the research leading to results presented here was conducted during a short  stay of GB at Institut Camille Jordan under project DYFICOLTI, ANR-13-BS01-0003-01 support. 

\bibliographystyle{plain}
\bibliography{referencesN}

\end{document}